\documentclass[english,11pt]{smfart}

\usepackage[utf8]{inputenc}
\usepackage{babel}

\usepackage[mono=false]{libertine}
\usepackage[T1]{fontenc}

\usepackage{amsthm}
\usepackage[cmintegrals,libertine]{newtxmath}
\usepackage[cal=dutchcal, calscaled=.95, scr=boondoxo]{mathalfa}
\useosf
\let\mathcal\mathbcal

\usepackage{microtype}

\usepackage{numprint}

\usepackage[a4paper,vmargin={3cm,3cm},hmargin={2.5cm,2.5cm}]{geometry}
\usepackage{tikz}
	\usetikzlibrary{calc}
	\usetikzlibrary{decorations.markings}

\usepackage{graphicx}

\usepackage[font=sf]{caption}
\usepackage{hyperref}
\hypersetup{
	colorlinks=true,
		citecolor=blue!60!black,
		linkcolor=red!60!black,
		urlcolor=green!40!black,
		filecolor=yellow!50!black,
	breaklinks=true,
	pdfpagemode=UseNone,
	bookmarksopen=false,
}

\usepackage{enumitem}
	\renewcommand{\theenumi}{{(\roman{enumi})}}
	\renewcommand{\labelenumi}{\theenumi}

\newtheorem{thm}{Theorem}
\newtheorem{lem}{Lemma}

\newtheorem{cor}{Corollary}

\theoremstyle{definition}
\newtheorem{rem}{Remark}

\newcommand{\ensembles}[1]{\mathbf{#1}}
	\newcommand{\N}{\ensembles{N}}
	\newcommand{\Z}{\ensembles{Z}}
	\newcommand{\R}{\ensembles{R}}
	\newcommand{\U}{\ensembles{U}}
	\renewcommand{\P}{\ensembles{P}}
	\newcommand{\E}{\ensembles{E}}
	\newcommand{\Var}{\mathrm{Var}}

\DeclareMathAlphabet{\mathbbo}{U}{bbold}{m}{n}
	\newcommand{\1}{\mathbbo{1}}
\newcommand{\ind}[1]{\1_{\{#1\}}}

\renewcommand{\Pr}[1]{\P\left(#1\right)}
\newcommand{\Prc}[2]{\P\left(#1 \;\middle|\; #2\right)}
\newcommand{\Es}[1]{\E\left[#1\right]}
\newcommand{\Esc}[2]{\E\left[#1 \;\middle|\; #2\right]}

\newcommand{\CRT}{\mathscr{T}}

\renewcommand{\d}{\mathrm{d}}
\newcommand{\e}{\mathrm{e}}
\newcommand{\Hsp}{H^{\mathrm{sp}}}
\newcommand{\tildeHsp}{\widetilde{H}^{\mathrm{sp}}}
\newcommand{\Csp}{C^{\mathrm{sp}}}
\newcommand{\Hspc}{\mathsf{H}^{\mathrm{sp}}}
\newcommand{\Xexc}{\mathscr{X}}
\newcommand{\Hexc}{\mathscr{H}}
\newcommand{\Snake}{\mathscr{S}}
\newcommand{\K}{\mathscr{K}}

\newcommand{\cv}[1][n]{\enskip\mathop{\longrightarrow}^{}_{#1 \to \infty}\enskip}
\newcommand{\cvloi}[1][n]{\enskip\mathop{\longrightarrow}^{(d)}_{#1 \to \infty}\enskip}

\newcommand{\cvproba}[1][n]{\enskip\mathop{\longrightarrow}^{\P}_{#1 \to \infty}\enskip}

\title{Scaling limits of discrete snakes with stable branching}

\author{Cyril \textsc{Marzouk}}
\address{Laboratoire de Math\'ematiques d'Orsay, Univ. Paris-Sud, CNRS, Universit\'e Paris-Saclay, 91405 Orsay, France.}
\email{cyril.marzouk@u-psud.fr}

\begin{document}

\begin{abstract}
We consider so-called discrete snakes obtained from size-conditioned critical Bienaymé--Galton--Watson trees by assigning to each node a random spatial position in such a way that the increments along each edge are i.i.d. When the offspring distribution belongs to the domain of attraction of a stable law with index $\alpha \in (1,2]$, we give a necessary and sufficient condition on the tail distribution of the spatial increments for this spatial tree to converge, in a functional sense, towards the Brownian snake driven by the $\alpha$-stable Lévy tree. We also study the case of heavier tails, and apply our result to study the number of inversions of a uniformly random permutation indexed by the tree. 
\end{abstract}

\maketitle

\begin{figure}[!ht] \centering
\includegraphics[width=.85\linewidth]{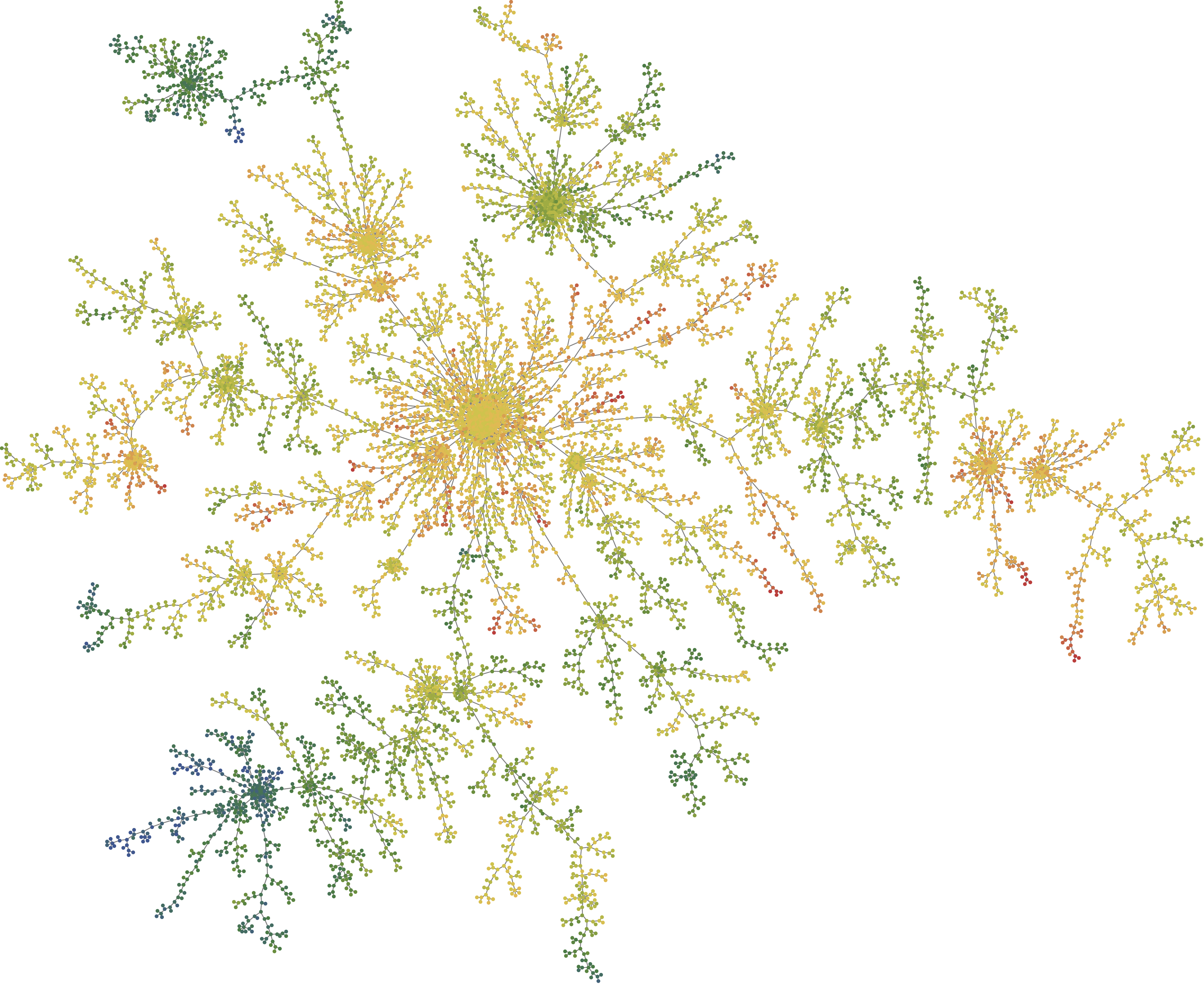}
\caption{A spatial stable Lévy tree with index $\alpha = \numprint{1.2}$; colours indicate the spatial position of each vertex, from blue for the lowest (negative) ones to green then yellow and finally red for the highest ones. It corresponds to Theorem~\ref{thm:convergence_serpent_iid} with $Y$ uniformly distributed on $\{-1, 0, 1\}$.}
\label{fig:arbre_couleurs}
\end{figure}

\section{Introduction and main results}

We investigate scaling limits of large size-conditioned random Bienaymé--Galton--Watson trees equipped with spatial positions, when the  offspring distribution belongs to the domain of attraction of a stable law. Our results extend previous ones established by Janson \& Marckert~\cite{Janson-Marckert:Convergence_of_discrete_snakes} when the offspring distribution admits finite exponential moments. Relaxing this strong assumption to even a finite variance hypothesis is often challenging, and our key result is a tight control on the geometry of the trees, which is of independent interest. Let us present precisely our main result, assuming some familiarities with Bienaymé--Galton--Watson trees and their coding by paths. The basic definitions are recalled in Section~\ref{sec:def_arbres_discrets_continus} below.

\subsection{Large Bienaymé--Galton--Watson trees}

Throughout this work, we fix a probability measure $\mu$ on $\Z_+ = \{0,1, \dots\}$ such that $\mu(0) > 0$ and $\sum_{k \ge 0} k \mu(k) = 1$. To simplify the exposition, we also assume that $\mu$ is aperiodic, in the sense that its support generates the whole group $\Z$, not just a strict subgroup; the general case only requires mild modifications. For every $n \ge 1$, we denote by $T_n$ a random plane tree distributed as a \emph{Bienaymé--Galton--Watson} tree with offspring distribution $\mu$ and conditioned to have $n+1$ vertices,\footnote{Our results also hold when the tree is conditioned to have $n+1$ vertices with out-degree in a fixed set $A \subset \Z_+$, appealing to \cite{Kortchemski:Invariance_principles_for_Galton_Watson_trees_conditioned_on_the_number_of_leaves}.} which is well defined for every $n$ large enough from the aperiodicity of $\mu$. It is well known that for every $a,b > 0$, $T_n$ has the same law as a random \emph{simply generated} tree with $n+1$ vertices, associated with the weight sequence $(a b^k \mu(k))_{k \ge 0}$, so there is almost no loss of generality to assume that $\mu$ has mean $1$.

Finally, we assume that there exists $\alpha \in (1,2]$ such that $\mu$ \emph{belongs to the domain of attraction of an $\alpha$-stable law}, which means that there exists an increasing sequence $(B_n)_{n \ge 1}$ such that if $(\xi_n)_{n \ge 1}$ is a sequence of i.i.d. random variables sampled from $\mu$, then $B_n^{-1} (\xi_1 + \dots + \xi_n - n)$ converges in distribution to a random variable $X^{(\alpha)}$ whose law is given by the Laplace exponent $\E[\exp(-\lambda X^{(\alpha)})] = \exp(\lambda^\alpha)$ for every $\lambda \ge 0$. Recall that $n^{-1/\alpha} B_n$ is \emph{slowly varying} at infinity and that if $\mu$ has variance $\sigma^2 \in (0,\infty)$, then this falls in the case $\alpha=2$ and we may take $B_n = \sqrt{n \sigma^2/2}$.

It is well-known that a planar tree can be encoded by discrete paths; in Section~\ref{sec:def_arbres_discrets_continus}, we recall the definition of the \emph{{\L}ukasiewicz path} $W_n$, the \emph{height process} $H_n$ and the \emph{contour process} $C_n$ associated with the tree $T_n$. Duquesne~\cite{Duquesne:A_limit_theorem_for_the_contour_process_of_conditioned_Galton_Watson_trees} (see also Kortchemski~\cite{Kortchemski:A_simple_proof_of_Duquesne_s_theorem_on_contour_processes_of_conditioned_Galton_Watson_trees, Kortchemski:Invariance_principles_for_Galton_Watson_trees_conditioned_on_the_number_of_leaves}) has proved that
\begin{equation}\label{eq:cv_GW_Duquesne}
\left(\frac{1}{B_n}W_n(n t), \frac{B_n}{n}H_n(n t), \frac{B_n}{n}C_n(2n t)\right)_{t \in [0,1]}
\cvloi
\left(\Xexc_t, \Hexc_t, \Hexc_t\right)_{t \in [0,1]}
\end{equation}
in the Skorokhod space $\mathscr{D}([0,1], \R^3)$, where $\Xexc$ is the normalised excursion of the $\alpha$-stable \emph{Lévy process} with no negative jump, whose value at time $1$ has the law of $X^{(\alpha)}$, and $\Hexc$ is the associated continuous \emph{height function}; see the references above for definitions and Figure~\ref{fig:Levy_hauteur} below for an illustration. In the case $\alpha=2$, the processes $\Xexc$ and $\Hexc$ are equal, both to $\sqrt{2}$ times the standard Brownian excursion. In any case, $\Hexc$ is a non-negative, continuous function, which vanishes only at $0$ and $1$. As any such function, it encodes a `continuous tree' called the \emph{$\alpha$-stable Lévy tree} $\CRT_\alpha$ of Duquesne, Le Gall \& Le Jan~\cite{Duquesne:A_limit_theorem_for_the_contour_process_of_conditioned_Galton_Watson_trees,Le_Gall-Le_Jan:Branching_processes_in_Levy_processes_the_exploration_process}, which generalises the celebrated Brownian tree of Aldous~\cite{Aldous:The_continuum_random_tree_3} in the case $\alpha=2$. The convergence~\eqref{eq:cv_GW_Duquesne} implies that the tree $T_n$, viewed as a metric space by endowing its vertex-set by the graph distance rescaled by a factor $\frac{B_n}{n}$, converges in distribution in the so-called Gromov--Hausdorff topology towards $\CRT_\alpha$, see e.g. Duquesne \& Le Gall~\cite{Duquesne-Le_Gall:Probabilistic_and_fractal_aspects_of_Levy_trees}.

\begin{figure}[!ht] \centering
\includegraphics[width=.45\linewidth]{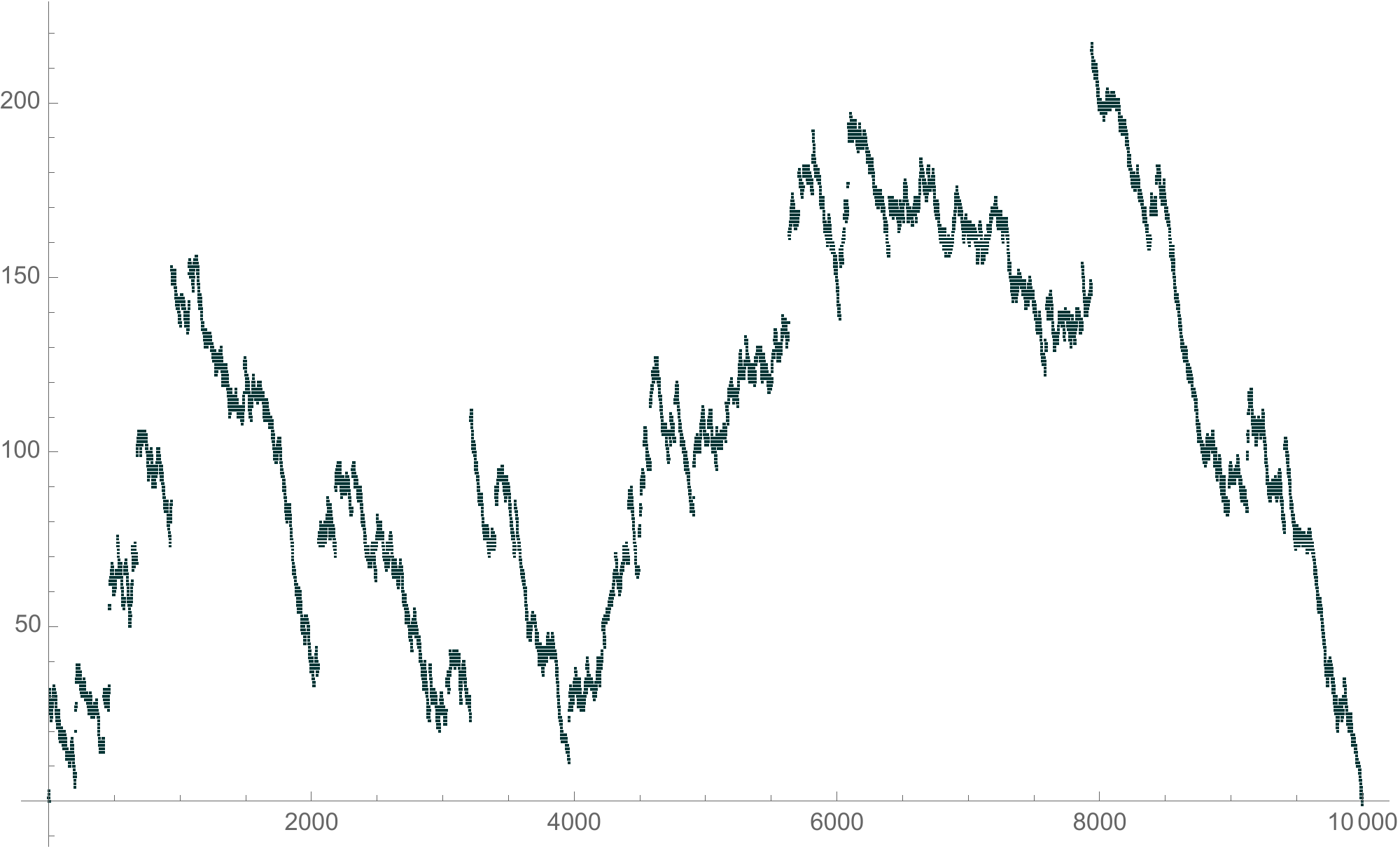}
\qquad
\includegraphics[width=.45\linewidth]{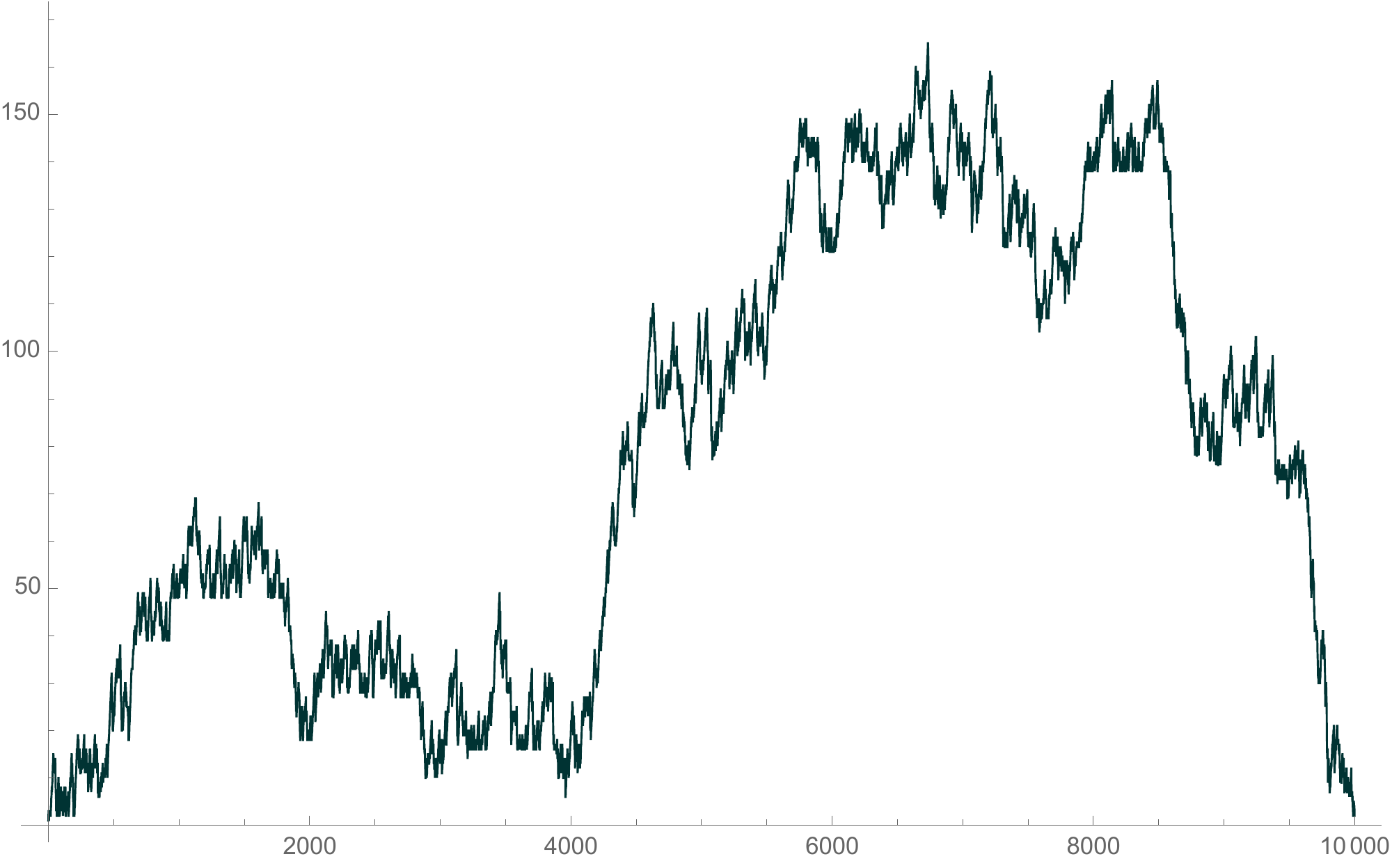}
\caption{The {\L}ukasiewicz path and the height process of $T_{\numprint{10000}}$ with $\alpha = \numprint{1.3}$.}
\label{fig:Levy_hauteur}
\end{figure}

\subsection{Spatial trees and applications}
In this paper, we consider \emph{spatial trees} (or \emph{labelled trees}, or \emph{discrete snakes}) which are plane trees in which each node $u$ of the tree $T$ carries a position $S_u$ in $\R$. We shall always assume that the root $\varnothing$ of the tree has position $S_\varnothing = 0$ by convention so the spatial positions $(S_u)_{u \in T}$ are entirely characterised by the \emph{displacements} $(Y_u)_{u \in T\setminus\{\varnothing\}}$. Several models of such random spatial trees have been studied and the simplest one is the following: let $Y$ be some random variable, then conditional on a random finite tree $T$, the spatial displacements $(Y_u)_{u \in T\setminus\{\varnothing\}}$ are i.i.d. copies of $Y$.

In the same way a tree $T_n$ with $n+1$ vertices is encoded by its height process $H_n$ and its contour process $C_n$, the spatial postions are encoded by the \emph{spatial height process} $\Hsp_n$ and the \emph{spatial contour process} $\Csp_n$. We consider scaling limits of these processes as $n \to \infty$. The most general such results are due to Janson \& Marckert~\cite{Janson-Marckert:Convergence_of_discrete_snakes} who considered the case where the tree $T_n$ is a size-conditioned Bienaymé--Galton--Watson tree whose offspring distribution has \emph{finite exponential moments}. All their results extend to our setting. The main one is a necessary and sufficient condition for the convergence towards the so-called \emph{Brownian snake} driven by the random excursion $\Hexc$, which, similarly to the discrete setting, is interpreted as a Brownian motion indexed by the stable tree $\CRT_\alpha$; see Section~\ref{sec:def_arbres_discrets_continus} below for a formal definition and Figure~\ref{fig:serpents_browniens} for two simulations.

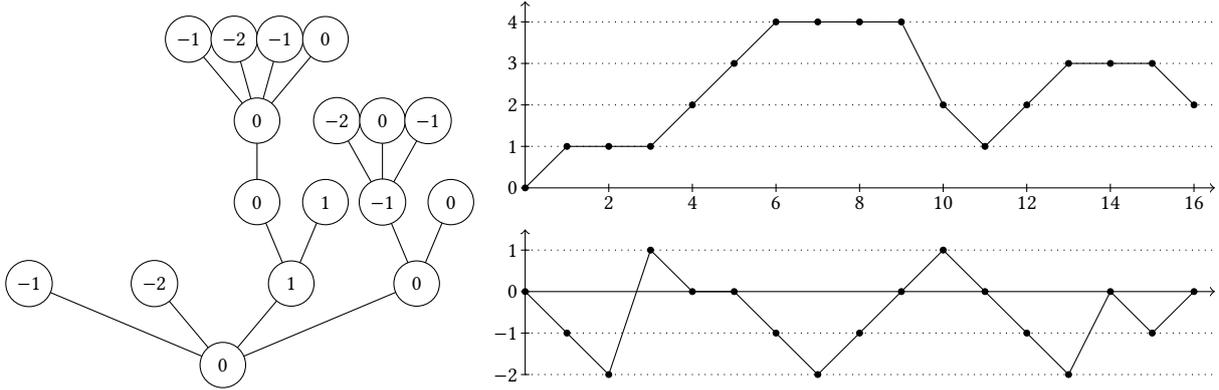
\begin{figure}[!ht]\centering
\def\r{.6}
\def\longueur{1.8}
\begin{tikzpicture}[scale=.6]
\coordinate (1) at (0,0*\longueur);
	\coordinate (2) at (-4.25,1*\longueur);
	\coordinate (3) at (-1.5,1*\longueur);
	\coordinate (4) at (1.5,1*\longueur);
		\coordinate (5) at (.75,2*\longueur);
			\coordinate (6) at (.75,3*\longueur);
				\coordinate (7) at (-.75,4*\longueur);
				\coordinate (8) at (.25,4*\longueur);
				\coordinate (9) at (1.25,4*\longueur);
				\coordinate (10) at (2.25,4*\longueur);
		\coordinate (11) at (2.25,2*\longueur);
	\coordinate (12) at (4.25,1*\longueur);
		\coordinate (13) at (3.5,2*\longueur);
			\coordinate (14) at (2.5,3*\longueur);
			\coordinate (15) at (3.5,3*\longueur);
			\coordinate (16) at (4.5,3*\longueur);
		\coordinate (17) at (5,2*\longueur);

\draw
	(1) -- (2)	(1) -- (3)	(1) -- (4)	(1) -- (12)
	(4) -- (5)	(4) -- (11)
	(5) -- (6)
	(6) -- (7)	(6) -- (8)	(6) -- (9)	(6) -- (10)
	(12) -- (13)	(12) -- (17)
	(13) -- (14)	(13) -- (15)	(13) -- (16)
;

\begin{scriptsize}
\node[circle, minimum size=\r cm, fill=white, draw] at (2) {$-1$};
\node[circle, minimum size=\r cm, fill=white, draw] at (3) {$-2$};
\node[circle, minimum size=\r cm, fill=white, draw] at (7) {$-1$};
\node[circle, minimum size=\r cm, fill=white, draw] at (8) {$-2$};
\node[circle, minimum size=\r cm, fill=white, draw] at (9) {$-1$};
\node[circle, minimum size=\r cm, fill=white, draw] at (10) {$0$};
\node[circle, minimum size=\r cm, fill=white, draw] at (11) {$1$};
\node[circle, minimum size=\r cm, fill=white, draw] at (14) {$-2$};
\node[circle, minimum size=\r cm, fill=white, draw] at (15) {$0$};
\node[circle, minimum size=\r cm, fill=white, draw] at (16) {$-1$};
\node[circle, minimum size=\r cm, fill=white, draw] at (17) {$0$};

\node[circle, minimum size=\r cm, fill=white, draw] at (1) {$0$};
\node[circle, minimum size=\r cm, fill=white, draw] at (4) {$1$};
\node[circle, minimum size=\r cm, fill=white, draw] at (5) {$0$};
\node[circle, minimum size=\r cm, fill=white, draw] at (6) {$0$};
\node[circle, minimum size=\r cm, fill=white, draw] at (12) {$0$};
\node[circle, minimum size=\r cm, fill=white, draw] at (13) {$-1$};
\end{scriptsize}
\end{tikzpicture}
\begin{scriptsize}
\begin{tikzpicture}[scale=.55]
\draw[thin, ->]	(0,0) -- (16.5,0);
\draw[thin, ->]	(0,0) -- (0,4.5);
\foreach \x in {1, 2, ..., 4}
	\draw[dotted]	(0,\x) -- (16.5,\x);
\foreach \x in {0, 1, ..., 4}
	\draw (.1,\x)--(-.1,\x)	(0,\x) node[left] {$\x$};
\foreach \x in {2, 4, ..., 16}
	\draw (\x,.1)--(\x,-.1)	(\x,0) node[below] {$\x$};
\coordinate (0) at (0, 0);
\coordinate (1) at (1, 1);
\coordinate (2) at (2, 1);
\coordinate (3) at (3, 1);
\coordinate (4) at (4, 2);
\coordinate (5) at (5, 3);
\coordinate (6) at (6, 4);
\coordinate (7) at (7, 4);
\coordinate (8) at (8, 4);
\coordinate (9) at (9, 4);
\coordinate (10) at (10, 2);
\coordinate (11) at (11, 1);
\coordinate (12) at (12, 2);
\coordinate (13) at (13, 3);
\coordinate (14) at (14, 3);
\coordinate (15) at (15, 3);
\coordinate (16) at (16, 2);
\newcommand{\lastx}{0}
\foreach \x [remember=\x as \lastx] in {1, 2, 3, ..., 16} \draw (\lastx) -- (\x);

\foreach \x in {0, 1, 2, 3, ..., 16} \draw [fill=black] (\x)	circle (2pt);

\begin{scope}[shift={(0,-2.5)}]
\draw[thin, ->]	(0,0) -- (16.5,0);
\draw[thin, ->]	(0,-2) -- (0,1.5);
\foreach \x in {-2, -1, 1}
	\draw[dotted]	(0,\x) -- (16.5,\x);
\foreach \x in {-2, -1, 0, 1}
	\draw (.1,\x)--(-.1,\x)	(0,\x) node[left] {$\x$};
%
\coordinate (0) at (0, 0);
\coordinate (1) at (1, -1);
\coordinate (2) at (2, -2);
\coordinate (3) at (3, 1);
\coordinate (4) at (4, 0);
\coordinate (5) at (5, 0);
\coordinate (6) at (6, -1);
\coordinate (7) at (7, -2);
\coordinate (8) at (8, -1);
\coordinate (9) at (9, 0);
\coordinate (10) at (10, 1);
\coordinate (11) at (11, 0);
\coordinate (12) at (12, -1);
\coordinate (13) at (13, -2);
\coordinate (14) at (14, 0);
\coordinate (15) at (15, -1);
\coordinate (16) at (16, 0);
\renewcommand{\lastx}{0}
\foreach \x [remember=\x as \lastx] in {1, 2, 3, ..., 16} \draw (\lastx) -- (\x);

\foreach \x in {0, 1, 2, 3, ..., 16} \draw [fill=black] (\x)	circle (2pt);
\end{scope}
\end{tikzpicture}
\end{scriptsize}
\caption{A spatial tree, its height process $H$ on top and its spatial height process $\Hsp$ below.}
\label{fig:codage_arbre_spatial}
\end{figure}

\begin{thm}[Convergence of discrete snakes]
\label{thm:convergence_serpent_iid}
Let $(\Hexc, \Snake)$ be the Brownian snake driven by the excursion $\Hexc$. Suppose $\E[Y]=0$ and $\Sigma^2 \coloneqq \E[Y^2] \in (0, \infty)$, then the following convergence in distribution holds in the sense of finite-dimensional marginals:
\[\left(\frac{B_n}{n} H_n(n t), \frac{B_n}{n} C_n(2n t), \left(\frac{B_n}{n \Sigma^2}\right)^{1/2} \Hsp_n(nt), \left(\frac{B_n}{n \Sigma^2}\right)^{1/2} \Csp_n(2nt)\right)_{t \in [0,1]} \cvloi (\Hexc_t, \Hexc_t, \Snake_t, \Snake_t)_{t \in [0,1]}.\]
It holds in $\mathscr{C}([0,1], \R^4)$ \emph{if and only if} $\P(|Y| \ge (n/B_n)^{1/2}) = o(n^{-1})$.
\end{thm}

In the finite-variance case $B_n = \sqrt{n \sigma^2/2}$, the last assumption is equivalent to $\P(|Y| \ge y) = o(y^{-4})$, which is slightly weaker than $\E[Y^4] < \infty$; otherwise, when the tree is less regular, one needs more regularity from the spatial displacements. 

Let us mention that general arguments show that $H_n$ and $C_n$, once rescaled, are close, see e.g. Duquesne \& Le Gall~\cite[Section~2.5]{Duquesne-Le_Gall:Random_trees_Levy_processes_and_spatial_branching_processes}, or Marckert \& Mokkadem~\cite{Marckert-Mokkadem:States_spaces_of_the_snake_and_its_tour_convergence_of_the_discrete_snake}. The same arguments apply for their spatial counterparts $\Hsp_n$ and $\Csp_n$ so we concentrate only on the joint convergence of $H_n$ and $\Hsp_n$.

\begin{figure}[!ht] \centering
\includegraphics[width=.45\linewidth]{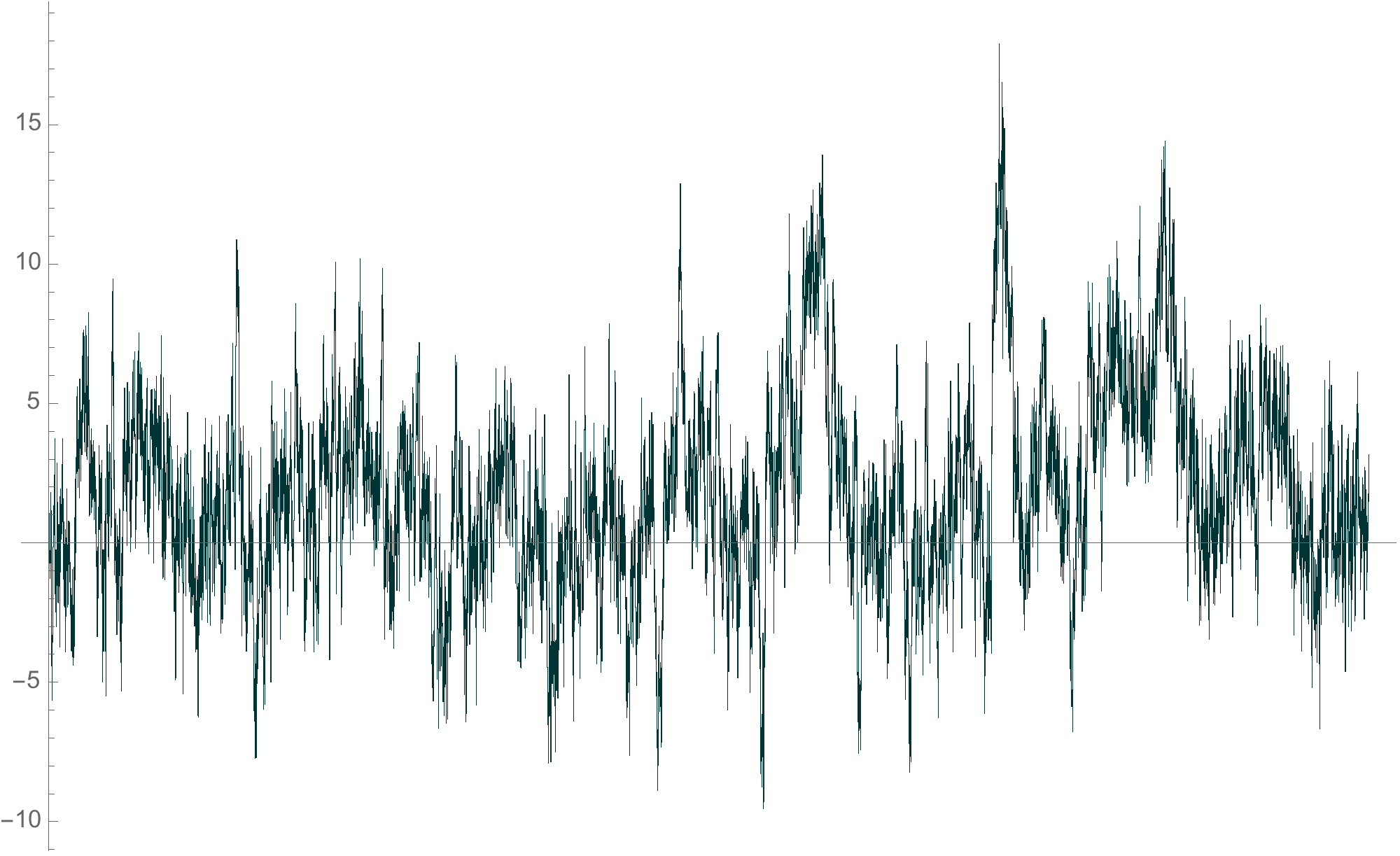}
\qquad
\includegraphics[width=.45\linewidth]{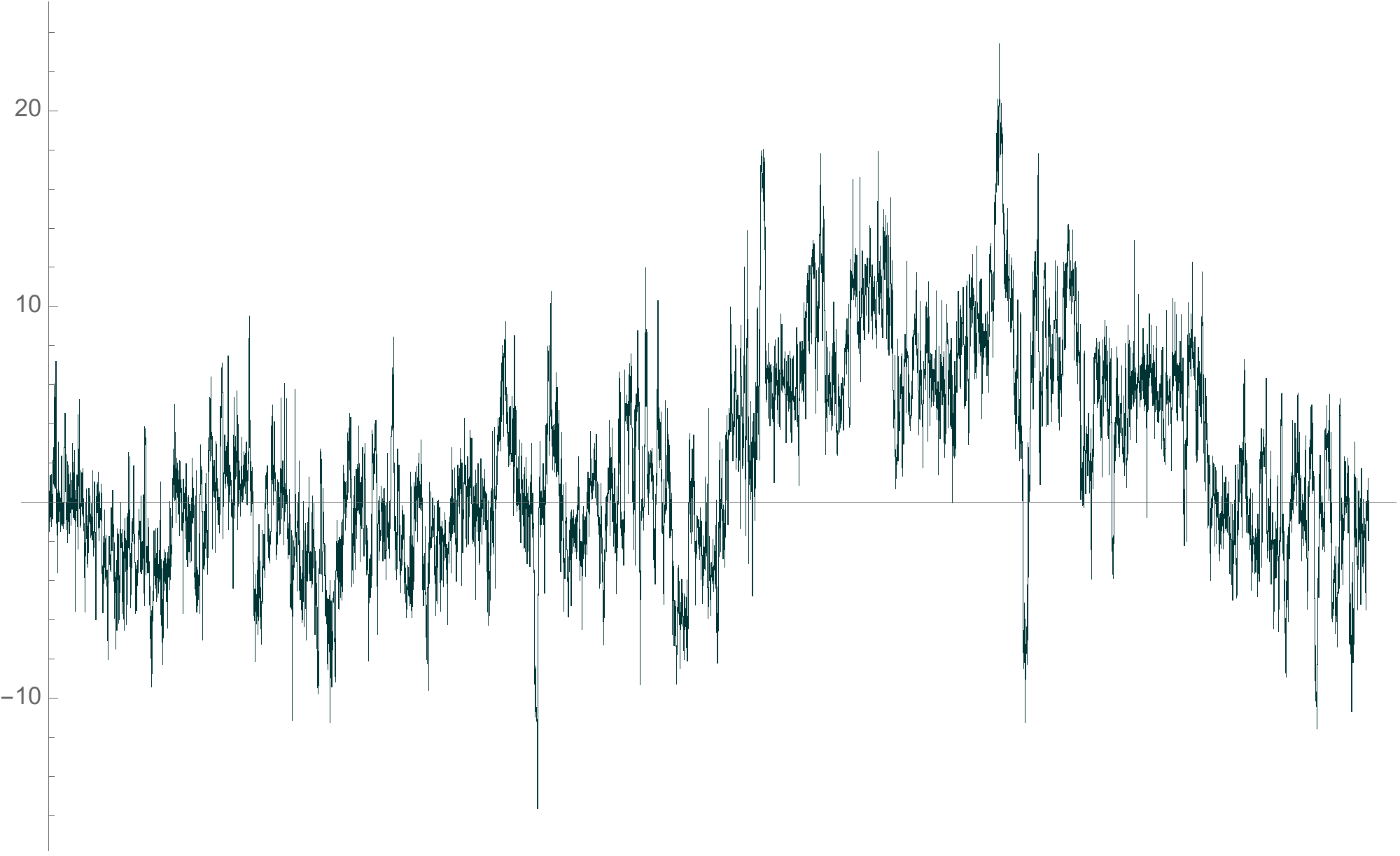}
\caption{Two instances of the spatial height process $\Sigma^{-1} H_n(n \cdot)$ associated with the height process of Figure~\ref{fig:Levy_hauteur}: on the left, $Y$ is uniformly distributed on $[-1, 1]$ and on the right, $Y$ is symmetric and such that $\P(Y > y) = \frac{1}{2} (1 + y)^{-10}$ so both satisfy Theorem~\ref{thm:convergence_serpent_iid}.}
\label{fig:serpents_browniens}
\end{figure}

Janson \& Marckert~\cite{Janson-Marckert:Convergence_of_discrete_snakes} also discuss the case of heavier tails, in which case the spatial processes converge once suitably rescaled towards a `hairy snake' with vertical peaks; statements are more involved and we defer them to Section~\ref{sec:queues_lourdes} below. Let us only mention the next result, which extends Theorem~8 in~\cite{Janson-Marckert:Convergence_of_discrete_snakes}.

\begin{thm}[Non centred snakes]
\label{thm:convergence_serpent_non_centre}
Suppose that $m \coloneqq \E[Y] \ne 0$. Then each process $\frac{B_n}{n} \Hsp_n(n\cdot)$ and $\frac{B_n}{n} \Csp_n(2n\cdot)$ is tight in $\mathscr{C}([0,1], \R)$ \emph{if and only if} $\P(|Y| \ge n/B_n) = o(n^{-1})$, and in this case we have the convergence in distribution in $\mathscr{C}([0,1], \R^4)$
\[\left(\frac{B_n}{n} H_n(n t), \frac{B_n}{n} C_n(2n t), \frac{B_n}{n} \Hsp_n(nt), \frac{B_n}{n} \Csp_n(2nt)\right)_{t \in [0,1]} \cvloi (\Hexc_t, \Hexc_t, m \cdot \Hexc_t, m \cdot \Hexc_t)_{t \in [0,1]}.\]
\end{thm}

Again, in the finite-variance case, the assumption is equivalent to $\P(|Y| \ge y) = o(y^{-2})$, which is slightly weaker than $\E[Y^2] < \infty$. Let us comment on the result when $Y \ge 0$ almost surely and $m > 0$. Then for every $u \in T_n$, the displacement $Y_u$ can be interpreted as the length of the edge from $u$ to its parent so $\Hsp_n$ and $\Csp_n$ can be interpreted as the height and contour processes of the tree $T_n$ with such random edge-lengths and Theorem~\ref{thm:convergence_serpent_non_centre} shows that this tree is close to the one obtained by assigning deterministic length $m$ to each edge of $T_n$, and it converges towards $m$ times the stable tree for the Gromov--Hausdorff topology, jointly with the original tree.

The main result of~\cite{Janson-Marckert:Convergence_of_discrete_snakes} has been used very recently by Cai \emph{et al.}~\cite{Cai-Holmgren-Janson-Johansson-Skerman:Inversions_in_split_trees_and_conditional_Galton_Watson_trees} to study the asymptotic number of \emph{inversions} in a random tree. Given the random tree $T_n$ with $n+1$ vertices listed $u_0, u_1, \dots, u_n$ and an independent uniformly random permutation of $\{0, \dots, n\}$, say, $\sigma$, assign the label $\sigma(i)$ to the vertex $u_i$ for every $i \in \{0, \dots, n\}$. The number of inversions of $T_n$ is then defined by
\[I(T_n) = \sum_{0 \le i < j \le n} \ind{u_i \text{ is a ancestor of } u_j} \ind{\sigma(i) > \sigma(j)}.\]
This extends the classical definition of the number of inversions of a permutation, when the tree contains a single branch. We refer to~\cite{Cai-Holmgren-Janson-Johansson-Skerman:Inversions_in_split_trees_and_conditional_Galton_Watson_trees} for a detailed review of the literature on this model. It is easy to see that $\E[I(T_n)\mid T_n]$ is half the so-called \emph{total path length} of $T_n$, whose asymptotic behaviour is well-understood. Cai \emph{et al.}~\cite[Theorem~5]{Cai-Holmgren-Janson-Johansson-Skerman:Inversions_in_split_trees_and_conditional_Galton_Watson_trees} study the fluctuations of $I(T_n)$ when $T_n$ is a size-conditioned Bienaymé--Galton--Watson tree whose offspring distribution admits exponential moments. Their argument is based on the convergence of snakes from~\cite{Janson-Marckert:Convergence_of_discrete_snakes} and extends readily as follows thanks to Theorem~\ref{thm:convergence_serpent_iid}.

\begin{cor}[Inversions on trees]
\label{cor:inversion}
We have the convergences in distribution
\[\frac{2B_n}{n^2} \E[I(T_n)\mid T_n] \cvloi \int_0^1 \Hexc_t \d t,
\quad\text{and}\quad
\left(\frac{B_n}{12n^3}\right)^{1/2} \left(I(T_n) - \E[I(T_n)\mid T_n]\right) \cvloi \int_0^1 \Snake_t \d t.\]
\end{cor}

Note that the scaling factors are respectively of order $n^{-3/2}$ and $n^{-5/4}$ in the finite-variance regime $B_n = \sqrt{n \sigma^2/2}$. When $\alpha=2$, recall that $\Hexc$ is $\sqrt{2}$ times the standard Brownian excursion; then the distribution of $2 \int_0^1 \Hexc_t \d t$ is known as the \emph{Airy distribution}; further, the random variable $\int_0^1 \Snake_t \d t$ is distributed as
\[\left(\int_{0 \le s < t \le 1} \min_{r \in [s,t]} \Hexc_r \d s \d t\right)^{1/2} \cdot \mathscr{N},\]
where $\mathscr{N}$ is standard Gaussian random variable independent of $\Hexc$; we refer to~\cite{Svante-Chassaing:The_center_of_mass_of_the_ISE_and_the_Wiener_index_of_trees} for more information on this random variable.

The main idea to prove tightness of spatial processes is to appeal to Kolmogorov's criterion, which enables one to avoid dealing with all the correlations between vertices. This requires a strong control on the geometry of the trees. Precisely, although the convergence~\eqref{eq:cv_GW_Duquesne} implies that the sequence $(\frac{B_n}{n} H_n(n \cdot))_{n \ge 1}$ is tight in $\mathscr{C}([0,1], \R)$, we need the following more precise estimate on the geometry of the trees.

\begin{lem}[H{\"o}lder norm of the height process]\label{lem:Holder_hauteur_GW_stables}
For every $\gamma \in (0, (\alpha-1)/\alpha)$, it holds that
\[\lim_{C \to \infty} \liminf_{n \to \infty} \Pr{\sup_{0 \le s \ne t \le 1} \frac{B_n}{n} \cdot \frac{|H_n(nt) - H_n(ns)|}{|t-s|^\gamma} \le C} = 1.\]
\end{lem}

By very different means, Gittenberger~\cite{Gittenberger:A_note_on} proved a similar statement for the contour function $C_n$, in the case $\alpha=2$, when the offspring distribution admits finite exponential moments\footnote{Even if the assumption is written as `finite variance' in~\cite{Gittenberger:A_note_on}, the proof actually requires exponential moments.} and Janson \& Marckert~\cite{Janson-Marckert:Convergence_of_discrete_snakes} built upon this result. Note that the maximal exponent $(\alpha-1)/\alpha$ corresponds to the maximal exponent for which the limit process $\Hexc$ is H{\"o}lder continuous, see~\cite[Theorem~1.4.4]{Duquesne-Le_Gall:Random_trees_Levy_processes_and_spatial_branching_processes}.

\subsection{More general models and random maps}

The initial motivation for studying spatial trees comes from the theory of \emph{random planar maps}. Indeed, the \emph{Schaeffer bijection} relates uniformly random \emph{quadrangulations} of the sphere with $n$ faces and such a model of spatial trees, when $\mu$ is the geometric distribution with parameter $1/2$, in which case $T_n$ has the uniform distribution amongst plane trees of size $n+1$, and when $Y$ has the uniform distribution on $\{-1, 0, 1\}$. The convergence of this particular spatial tree has been obtained Chassaing \& Schaeffer~\cite{Chassaing-Schaeffer:Random_planar_lattices_and_integrated_super_Brownian_excursion}. More general models of random maps are also related to spatial trees, via the \emph{Bouttier--Di Francesco--Guitter} bijection \cite{Bouttier-Di_Francesco-Guitter:Planar_maps_as_labeled_mobiles} and the \emph{Janson-Stef{\'a}nsson} bijection \cite{Janson-Stefansson:Scaling_limits_of_random_planar_maps_with_a_unique_large_face}; however, in this case, the displacements are neither independent nor identically distributed. Analogous convergences to Theorem~\ref{thm:convergence_serpent_iid} in this case have been proved by Marckert \& Mokkadem~\cite{Marckert-Mokkadem:States_spaces_of_the_snake_and_its_tour_convergence_of_the_discrete_snake} still for the uniform random trees, but for general displacements, under a `$(8+\varepsilon)$-moment' assumption; Gittenberger~\cite{Gittenberger:A_note_on} extended this result to the case where $\mu$ has finite exponential moments, and then Marckert \& Miermont~\cite{Marckert-Miermont:Invariance_principles_for_random_bipartite_planar_maps} reduced the assumptions on the displacements to a `$(4+\varepsilon)$-moment', see also Miermont~\cite{Miermont:Invariance_principles_for_spatial_multitype_Galton_Watson_trees} for similar results on multi-type Bienaymé--Galton--Watson trees, Marckert~\cite{Marckert:The_lineage_process_in_Galton_Watson_trees_and_globally_centered_discrete_snakes} for `globally centred' displacements, and finally~\cite{Marzouk:Scaling_limits_of_random_bipartite_planar_maps_with_a_prescribed_degree_sequence} for trees (more general than size-conditioned Bienaymé--Galton--Watson trees) with finite variance, but only for the very particular displacements associated with maps. Appealing to Lemma~\ref{lem:Holder_hauteur_GW_stables}, it seems that the `$(4+\varepsilon)$-moment' assumption suffices in the case where $\mu$ belongs to the domain of attraction of a Gaussian law to ensure the convergence towards $(\Hexc, \Snake)$. However in the $\alpha$-stable case with $\alpha < 2$, the limit may be different and depend more precisely on the displacements, see Le Gall \& Miermont~\cite{Le_Gall-Miermont:Scaling_limits_of_random_planar_maps_with_large_faces}, again for the very particular displacements associated with maps.

\subsection{Techniques}

The rest of this paper is organised as follows: In Section~\ref{sec:BGW}, we first recall the coding of plane trees by paths and define the limit object of interest $\Snake$; after recalling a few results on slowly varying functions and well-known results on Bienaymé--Galton--Watson trees, we prove Lemma~\ref{lem:Holder_hauteur_GW_stables}. The idea is to rely on the {\L}ukasiewicz path of the tree, since height of vertices corresponds to positive records of the latter, which is an excursion of a left-continuous random walk in the domain of attraction of a stable law, so it already has attracted a lot of attention and we may use several existing results, such as those due to Doney~\cite{Doney:On_the_exact_asymptotic_behaviour_of_the_distribution_of_ladder_epochs}.

In Section~\ref{sec:serpents}, we prove Theorem~\ref{thm:convergence_serpent_iid}, Theorem~\ref{thm:convergence_serpent_non_centre} and Corollary~\ref{cor:inversion}. The proof of the two theorems follows the ideas of Janson \& Marckert~\cite{Janson-Marckert:Convergence_of_discrete_snakes} which are quite general once we have Lemma~\ref{lem:Holder_hauteur_GW_stables}. However, several technical adaptations are needed here to deal with the heavier tails for the offspring distribution. Finally, in Section~\ref{sec:queues_lourdes}, we state and prove results on the convergence of similar to Theorem~\ref{thm:convergence_serpent_iid} when $Y$ has heavier, regularly varying tails. Again, the proof scheme follows that of~\cite{Janson-Marckert:Convergence_of_discrete_snakes} but requires technical adaptation.

\subsection*{Acknowledgment}
I wish to thank Nicolas Curien for a stimulating discussion on the proof of Lemma~\ref{lem:Holder_hauteur_GW_stables} when I started to have some doubts on the strategy used below. Many thanks are also due to Igor Kortchemski who spotted a mistake in a first draft.

This work was supported by a public grant as part of the Fondation Mathématique Jacques Hadamard.

\section{Geometry of large Bienaymé--Galton--Watson trees}
\label{sec:BGW}

\subsection{Discrete and continuous snakes}
\label{sec:def_arbres_discrets_continus}

We follow the notation of Neveu~\cite{Neuveu:Arbres_et_processus_de_Galton_Watson} and view discrete trees as words. Let $\N = \{1, 2, \dots\}$ be the set of all positive integers, set $\N^0 = \{\varnothing\}$ and consider the set $\U = \bigcup_{n \ge 0} \N^n$. For every $u = (u_1, \dots, u_n) \in \U$, we denote by $|u| = n$ the length of $u$; if $n \ge 1$, we define its \emph{prefix} $pr(u) = (u_1, \dots, u_{n-1})$ and we let $\chi_u = u_n$; for $v = (v_1, \dots, v_m) \in \U$, we let $uv = (u_1, \dots, u_n, v_1, \dots, v_m) \in \U$ be the concatenation of $u$ and $v$. We endow $\U$ with the \emph{lexicographical order}: given $u, v \in \U$, let $w \in \U$ be their longest common prefix, that is $u = w(u_1, \dots, u_n)$, $v = w(v_1, \dots, v_m)$ and $u_1 \ne v_1$, then $u < v$ if $u_1 < v_1$.

A (plane) \emph{tree} is a non-empty, finite subset $T \subset \U$ such that:
\begin{enumerate}
\item $\varnothing \in T$;
\item if $u \in T$ with $|u| \ge 1$, then $pr(u) \in T$;
\item if $u \in T$, then there exists an integer $k_u \ge 0$ such that $ui \in T$ if and only if $1 \le i \le k_u$.
\end{enumerate}

We shall view each vertex $u$ of a tree $T$ as an individual of a population for which $T$ is the genealogical tree. The vertex $\varnothing$ is called the \emph{root} of the tree and for every $u \in T$, $k_u$ is the number of \emph{children} of $u$ (if $k_u = 0$, then $u$ is called a \emph{leaf}, otherwise, $u$ is called an \emph{internal vertex}) and $u1, \dots, uk_u$ are these children from left to right, $\chi_u$ is the relative position of $u$ among its siblings, $|u|$ is its \emph{generation}, $pr(u)$ is its \emph{parent} and more generally, the vertices $u, pr(u), pr \circ pr (u), \dots, pr^{|u|}(u) = \varnothing$ are its \emph{ancestors}; the longest common prefix of two elements is their \emph{last common ancestor}. We shall denote by $\llbracket u , v \rrbracket$ the unique non-crossing path between $u$ and $v$.

Fix a tree $T$ with $n+1$ vertices, listed $\varnothing = u_0 < u_1 < \dots < u_n$ in lexicographical order. We describe three discrete paths which each encode $T$. First, its \emph{{\L}ukasiewicz path} $W = (W(j) ; 0 \le j \le n+1)$ is defined by $W(0) = 0$ and for every $0 \le j \le n$,
\[W(j+1) = W(j) + k_{u_j}-1.\]
One easily checks that $W(j) \ge 0$ for every $0 \le j \le n$ but $W(n+1)=-1$. Next, we define the \emph{height process} $H = (H(j); 0 \le j \le n)$ by setting for every $0 \le j \le n$,
\[H(j) = |u_j|.\]
Finally, define the \emph{contour sequence} $(c_0, c_1, \dots ,c_{2n})$ of $T$ as follows: $c_0 = \varnothing$ and for each $i \in \{0, \dots, 2n-1\}$, $c_{i+1}$ is either the first child of $c_i$ which does not appear in the sequence $(c_0, \dots, c_i)$, or the parent of $c_i$ if all its children already appear in this sequence. The lexicographical order on the tree corresponds to the \emph{depth-first search order}, whereas the contour order corresponds to `moving around the tree in clockwise order'. The \emph{contour process} $C = (C(j) ; 0 \le j \le 2n)$ is defined by setting for every $0 \le j \le 2n$,
\[C(j) = |c_j|.\]
We refer to Figure~\ref{fig:codage_arbre} for an illustration of these functions.

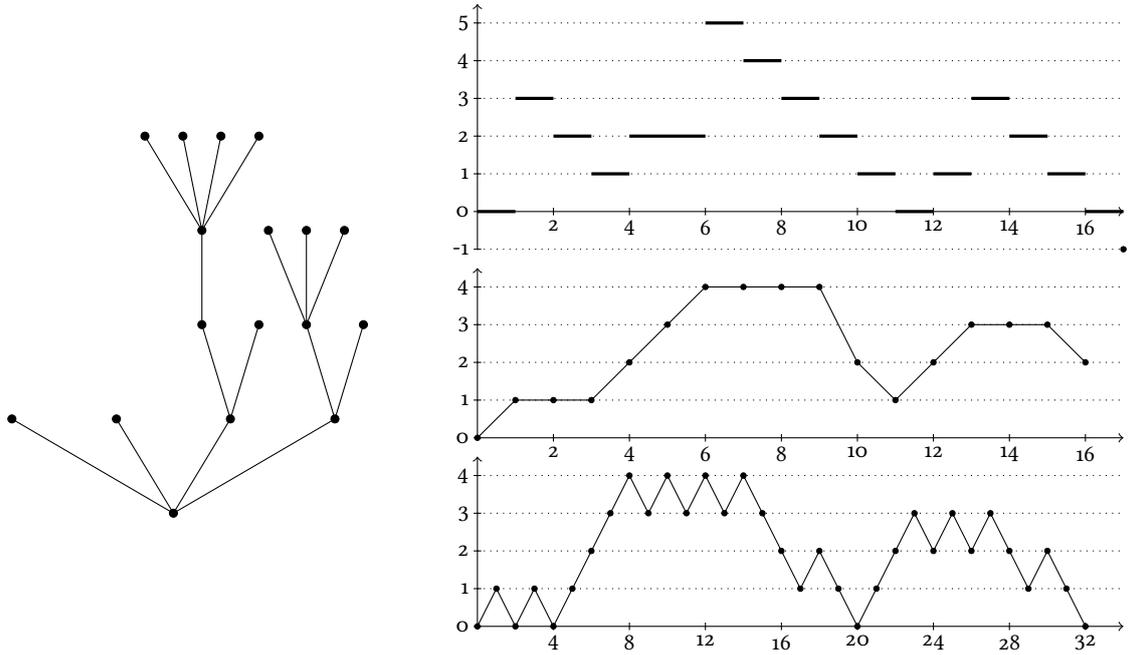
\begin{figure}[!ht]\centering
\def\r{.6}
\def\longueur{2.5}
\begin{tikzpicture}[scale=.5]
\coordinate (1) at (0,0*\longueur);
	\coordinate (2) at (-4.25,1*\longueur);
	\coordinate (3) at (-1.5,1*\longueur);
	\coordinate (4) at (1.5,1*\longueur);
		\coordinate (5) at (.75,2*\longueur);
			\coordinate (6) at (.75,3*\longueur);
				\coordinate (7) at (-.75,4*\longueur);
				\coordinate (8) at (.25,4*\longueur);
				\coordinate (9) at (1.25,4*\longueur);
				\coordinate (10) at (2.25,4*\longueur);
		\coordinate (11) at (2.25,2*\longueur);
	\coordinate (12) at (4.25,1*\longueur);
		\coordinate (13) at (3.5,2*\longueur);
			\coordinate (14) at (2.5,3*\longueur);
			\coordinate (15) at (3.5,3*\longueur);
			\coordinate (16) at (4.5,3*\longueur);
		\coordinate (17) at (5,2*\longueur);

\draw
	(1) -- (2)	(1) -- (3)	(1) -- (4)	(1) -- (12)
	(4) -- (5)	(4) -- (11)
	(5) -- (6)
	(6) -- (7)	(6) -- (8)	(6) -- (9)	(6) -- (10)
	(12) -- (13)	(12) -- (17)
	(13) -- (14)	(13) -- (15)	(13) -- (16)
;

\foreach \x in {1, 2, ..., 17}
\draw[fill=black] (\x) circle (3pt);
\begin{footnotesize}
\begin{scope}[shift={(8,8)}]
\draw[thin, ->]	(0,0) -- (17,0);
\draw[thin, ->]	(0,-1) -- (0,5.5);
\foreach \x in {-1, 1, 2, 3, 4, 5}
	\draw[dotted]	(0,\x) -- (17,\x);
\foreach \x in {-1, 0, 1, ..., 5}
	\draw (.1,\x)--(-.1,\x)	(0,\x) node[left] {\x};
\foreach \x in {2, 4, ..., 16}
	\draw (\x,.1)--(\x,-.1)	(\x,0) node[below] {\x};

\draw[very thick]
	(0, 0) -- ++ (1,0)
	++(0,3) -- ++ (1,0)
	++(0,-1) -- ++ (1,0)
	++(0,-1) -- ++ (1,0)
	++(0,1) -- ++ (1,0)
	++(0,0) -- ++ (1,0)
	++(0,3) -- ++ (1,0)
	++(0,-1) -- ++ (1,0)
	++(0,-1) -- ++ (1,0)
	++(0,-1) -- ++ (1,0)
	++(0,-1) -- ++ (1,0)
	++(0,-1) -- ++ (1,0)
	++(0,1) -- ++ (1,0)
	++(0,2) -- ++ (1,0)
	++(0,-1) -- ++ (1,0)
	++(0,-1) -- ++ (1,0)
	++(0,-1) -- ++ (1,0)
;
\draw[fill=black] (17,-1) circle (2pt);
\end{scope}
\begin{scope}[shift={(8,2)}]
\draw[thin, ->]	(0,0) -- (17,0);
\draw[thin, ->]	(0,0) -- (0,4.5);
\foreach \x in {1, 2, ..., 4}
	\draw[dotted]	(0,\x) -- (17,\x);
\foreach \x in {0, 1, ..., 4}
	\draw (.1,\x)--(-.1,\x)	(0,\x) node[left] {\x};
\foreach \x in {2, 4, ..., 16}
	\draw (\x,.1)--(\x,-.1)	(\x,0) node[below] {\x};

\draw[fill=black]
	(0, 0) circle (2pt) --
	++ (1, 1) circle (2pt) --
	++ (1, 0) circle (2pt) --
	++ (1, 0) circle (2pt) --
	++ (1, 1) circle (2pt) --
	++ (1, 1) circle (2pt) --
	++ (1, 1) circle (2pt) --
	++ (1, 0) circle (2pt) --
	++ (1, 0) circle (2pt) --
	++ (1, 0) circle (2pt) --
	++ (1, -2) circle (2pt) --
	++ (1, -1) circle (2pt) --
	++ (1, 1) circle (2pt) --
	++ (1, 1) circle (2pt) --
	++ (1, 0) circle (2pt) --
	++ (1, 0) circle (2pt) --
	++ (1, -1) circle (2pt)
;
\end{scope}
\begin{scope}[shift={(8,-3)}]
\draw[thin, ->]	(0,0) -- (17,0);
\draw[thin, ->]	(0,0) -- (0,4.5);
\foreach \x in {1, 2, ..., 4}
	\draw[dotted]	(0,\x) -- (17,\x);
\foreach \x in {0, 1, ..., 4}
	\draw (.1,\x)--(-.1,\x)	(0,\x) node[left] {\x};
\foreach \x in {4, 8, ..., 32}
	\draw (.5*\x,.1)--(.5*\x,-.1)	(.5*\x,0) node[below] {\x};

\draw[fill=black]
	(0, 0) circle (2pt) --
	++ (.5, 1) circle (2pt) --
	++ (.5, -1) circle (2pt) --
	++ (.5, 1) circle (2pt) --
	++ (.5, -1) circle (2pt) --
	++ (.5, 1) circle (2pt) --
	++ (.5, 1) circle (2pt) --
	++ (.5, 1) circle (2pt) --
	++ (.5, 1) circle (2pt) --
	++ (.5, -1) circle (2pt) --
	++ (.5, 1) circle (2pt) --
	++ (.5, -1) circle (2pt) --
	++ (.5, 1) circle (2pt) --
	++ (.5, -1) circle (2pt) --
	++ (.5, 1) circle (2pt) --
	++ (.5, -1) circle (2pt) --
	++ (.5, -1) circle (2pt) --
	++ (.5, -1) circle (2pt) --
	++ (.5, 1) circle (2pt) --
	++ (.5, -1) circle (2pt) --
	++ (.5, -1) circle (2pt) --
	++ (.5, 1) circle (2pt) --
	++ (.5, 1) circle (2pt) --
	++ (.5, 1) circle (2pt) --
	++ (.5, -1) circle (2pt) --
	++ (.5, 1) circle (2pt) --
	++ (.5, -1) circle (2pt) --
	++ (.5, 1) circle (2pt) --
	++ (.5, -1) circle (2pt) --
	++ (.5, -1) circle (2pt) --
	++ (.5, 1) circle (2pt) --
	++ (.5, -1) circle (2pt) --
	++ (.5, -1) circle (2pt)
;
\end{scope}
\end{footnotesize}
\end{tikzpicture}
\caption{A tree on the left, and on the right, from top to bottom: its {\L}ukasiewicz path $W$, its height process $H$, and its contour process $C$.}
\label{fig:codage_arbre}
\end{figure}

A \emph{spatial tree} $(T, (S_u; u \in T))$ is a tree $T$ in which each individual $u$ is assigned a spatial position $S_u \in \R$, with $S_\varnothing = 0$. We encode these positions via the \emph{spatial height} and \emph{spatial contour} processes $\Hsp$ and $\Csp$ respectively, defined by $\Hsp(j) = S_{u_j}$ for every $0 \le j \le n$ and $\Csp(j) = S_{c_j}$ for every $0 \le j \le 2n$, where $n$ is the number of edges of the tree. See Figure~\ref{fig:codage_arbre_spatial} for an illustration of $\Hsp$.

~
\begin{center}
\begin{minipage}{.85\linewidth}
\emph{Without further notice, throughout this work, every {\L}ukasiewicz path shall be viewed as a step function, jumping at integer times, whereas height and contour processes, as well as their spatial versions, shall be viewed as continuous functions after interpolating linearly between integer times.}
\end{minipage}
\end{center}
~

The analogous continuous objects we shall consider are the \emph{stable Lévy tree} of Duquesne, Le Gall and Le Jan~\cite{Duquesne:A_limit_theorem_for_the_contour_process_of_conditioned_Galton_Watson_trees,Le_Gall-Le_Jan:Branching_processes_in_Levy_processes_the_exploration_process} which generalise Aldous' Brownian Continuum Random Tree~\cite{Aldous:The_continuum_random_tree_3} in the case $\alpha=2$. Recall that $\Hexc = (\Hexc_t ; t \in [0,1])$ denotes the excursion of the \emph{height process} associated with the $\alpha$-stable Lévy process with no negative jump; we shall not need the precise definition of this process but we refer the reader to~\cite[Section~3.1 and~3.2]{Duquesne:A_limit_theorem_for_the_contour_process_of_conditioned_Galton_Watson_trees}. For every $s, t \in [0,1]$, set
\[d_\Hexc(s,t) = \Hexc_s + \Hexc_t - 2 \min_{r \in [s \wedge t, s \vee t]} \Hexc_r.\]
One easily checks that $d_\Hexc$ is a random pseudo-metric on $[0,1]$, we then define an equivalence relation on $[0,1]$ by setting $s \sim_\Hexc t$ whenever $d_\Hexc(s,t)=0$. Consider the quotient space $\CRT_\alpha = [0,1] / \sim_\Hexc$, we let $\pi_\Hexc$ be the canonical projection $[0,1] \to \CRT_\alpha$; then $d_\Hexc$ induces a metric on $\CRT_\alpha$ that we still denote by $d_\Hexc$. The space $(\CRT_\alpha, d_\Hexc)$ is a so-called compact real-tree, naturally rooted at $\pi_\Hexc(0) = \pi_\Hexc(1)$, called the \emph{stable tree} coded by $\Hexc$.

We construct another process $\Snake = (\Snake_t ; t \in [0,1])$ on the same probability space as $\Hexc$ which, conditional on $\Hexc$, is a centred Gaussian process satisfying for every $0 \le s \le t \le 1$,
\[\Esc{|\Snake_s - \Snake_t|^2}{\Hexc} = d_\Hexc(s,t)
\qquad\text{or, equivalently,}\qquad
\Esc{\Snake_s \Snake_t}{\Hexc} = \min_{r \in [s, t]} \Hexc_r.\]
Observe that, almost surely, $\Snake_0=0$ and $\Snake_s = \Snake_t$ whenever $s \sim_\Hexc t$ so $\Snake$ can be seen as a Brownian motion indexed by $\CRT_\alpha$ by setting $\Snake_{\pi_\Hexc(t)} = \Snake_t$ for every $t \in [0,1]$. We interpret $\Snake_x$ as the spatial position of an element $x \in \CRT_\alpha$; the pair $(\CRT_\alpha, (\Snake_x; x \in \CRT_\alpha))$ is a continuous analog of spatial plane trees. 

The \emph{Brownian snake} driven by $\Hexc$~\cite{Le_Gall:Nachdiplomsvorlesung,Duquesne-Le_Gall:Random_trees_Levy_processes_and_spatial_branching_processes} is a path-valued process which associates with each time $t \in [0,1]$ the hole path of values $\Snake_x$ where $x$ ranges over all the ancestors of $\pi_\Hexc(t)$ in $\CRT_\alpha$, from the root to $\pi_\Hexc(t)$, so the process $\Snake$ that we consider is only its `tip', which is called the \emph{head of the Brownian snake}. In this work we only consider the head of the snakes, which is in principle different than the entire snakes; nevertheless, Marckert \& Mokkadem~\cite{Marckert-Mokkadem:States_spaces_of_the_snake_and_its_tour_convergence_of_the_discrete_snake} proved a homeomorphism theorem which translates one into the other. Theorem~\ref{thm:convergence_serpent_iid} then implies the convergence of the whole snake towards the Brownian snake, see~\cite[Corollary~2]{Janson-Marckert:Convergence_of_discrete_snakes}.

It is known, see, e.g.~\cite[Chapter~IV.4]{Le_Gall:Nachdiplomsvorlesung} on the whole Brownian snake, that the pair $(\Hexc, \Snake)$ admits a continuous version and, without further notice, we shall work throughout this paper with this version.

\subsection{Bienaymé--Galton--Watson trees and random walks}

Recall that $\mu$ is a probability measure on $\Z_+$ satisfying a few assumptions given in the introduction. The Bienaymé--Galton--Watson distribution is the law on the set of all finite plane trees, which gives mass $\prod_{u \in T} \mu(k_u)$ to every such tree $T$. We then denote by $T_n$ such a random tree conditioned to have $n+1$ vertices. 

The key to prove Lemma~\ref{lem:Holder_hauteur_GW_stables} is a well-known relation between the height process $H_n$ and the {\L}ukasiewicz path $W_n$, as well as a representation of the latter from a random walk. Our argument is inspired by the work of Le Gall \& Miermont~\cite[Proof of Lemma~6 and~7]{Le_Gall-Miermont:Scaling_limits_of_random_planar_maps_with_large_faces} who consider an infinite forest of unconditioned trees, which is slightly easier thanks to the fact that the {\L}ukasiewicz path is then a non-conditioned random walk; furthermore, there it is supposed that $\mu([k, \infty)) \sim c k^{-\alpha}$ for some constant $c > 0$, which is a stronger assumption that ours, and several arguments do not carry over.

\subsubsection{On slowly varying functions and domains of attraction}
\label{sec:variation_lente}

Let us present a few prerequisites on slowly varying functions. First, recall that a function $l : [0, \infty) \to \R$ is said to be slowly varying (at infinity) when for every $c > 0$, it holds that
\[\lim_{x \to \infty} \frac{l(cx)}{l(x)} = 1.\]
A property of slowly varying functions that we shall used repeatedly in Section~\ref{sec:serpents} and~\ref{sec:queues_lourdes} is that for every $\varepsilon > 0$, it holds that
\[\lim_{x \to \infty} x^{-\varepsilon} l(x) = 0,
\qquad\text{and}\qquad
\lim_{x \to \infty} x^\varepsilon l(x) = \infty,\]
see e.g. Seneta's book~\cite{Seneta:Regularly_varying_functions} for more information on slowly varying functions (see Chapter~1.5 there for this property).

Let us fix a random variable $X$ on $\{-1, 0, 1, \dots\}$ with law $\P(X = k) = \mu(k+1)$ for every $k \ge -1$, so $\E[X]=0$. Since $\mu$ belongs to the domain of attraction of a stable law with index $\alpha \in (1,2]$, there exists two slowly varying functions $L$ and $L_1$ such that for every $n \ge 1$,
\[\Es{X^2 \ind{X \le n}} = n^{2-\alpha} L(n)
\qquad\text{and}\qquad
\Pr{X \ge n} = n^{-\alpha} L_1(n).\]
The two functions are related by
\[\lim_{n \to \infty} \frac{L_1(n)}{L(n)} 
= \lim_{n \to \infty} \frac{n^2 \P(X \ge n)}{\E[X^2 \ind{X \le n}]} 
= \frac{2-\alpha}{\alpha},\]
see Feller~\cite[Chapter~XVII, Equation 5.16]{Feller:An_introduction_to_probability_theory_and_its_applications_Volume_2}. 
We shall need a third slowly varying function $L^\ast$ (see Doney~\cite[Equation~2.2]{Doney:On_the_exact_asymptotic_behaviour_of_the_distribution_of_ladder_epochs}), defined uniquely up to asymptotic equivalence as the \emph{conjugate} of $1/L$ by the following equivalent asymptotic relations:
\[L(x)^{-1/\alpha} L^\ast(x^\alpha L(x)^{-1}) \cv[x] 1
\qquad\text{and}\qquad
L^\ast(x)^{-\alpha} L(x^{1/\alpha} L^\ast(x)) \cv[x] 1.\]
We refer to~\cite[Chapter~1.6]{Seneta:Regularly_varying_functions} for more information about conjugation of slowly varying functions. Let $S = (S(n))_{n \ge 0}$ be a random walk started from $0$ with step distribution $X$. As recalled in the introduction, there exists an increasing sequence $(B_n)_{n \ge 1}$ such that if $(X_n)_{n \ge 1}$ are i.d.d. copies of $X$, then $B_n^{-1} S(n)$ converges in distribution to some $\alpha$-stable random variable. The sequence $\ell(n) = n^{-1/\alpha} B_n$ is slowly varying at infinity and in fact, the ratio $L^\ast(n)/\ell(n)$ converges to some positive and finite limit. For $\alpha < 2$, this was observed by Doney~\cite{Doney:On_the_exact_asymptotic_behaviour_of_the_distribution_of_ladder_epochs}, but it extends to the case $\alpha=2$, see the remark between Equation~2.2 and Theorem~1 in~\cite{Doney:On_the_exact_asymptotic_behaviour_of_the_distribution_of_ladder_epochs}: the function $L$ there is $1/L$ here. By comparing the preceding asymptotic relations between $L$ and $L^\ast$ to \cite[Equation~7]{Kortchemski:Sub_exponential_tail_bounds_for_conditioned_stable_Bienayme_Galton_Watson_trees}, one gets precisely
\[\lim_{x \to \infty} \frac{L^\ast(x)}{\ell(x)} = \frac{1}{(2-\alpha) \Gamma(-\alpha)},\]
where, by continuity, the limit is interpreted as equal to $2$ if $\alpha=2$.

Doney~\cite[Theorem~1]{Doney:On_the_exact_asymptotic_behaviour_of_the_distribution_of_ladder_epochs} studies the behaviour of the strict record times of the walk $S$, but his work extends \emph{mutatis mutandis} to weak record times: let $\tau_0 = 0$ and for every $i \ge 1$, let $\tau_i = \inf\{k > \tau_{i-1} : S(k) \ge S(\tau_{i-1})\}$; in other words, the times $(\tau_n)_{n \ge 0}$ list those $k \ge 0$ such that $S(k) = \max_{0 \le i \le k} S(i)$. Then the random variables $(\tau_{n+1} - \tau_n)_{n \ge 0}$ are i.d.d. and according to~\cite[Theorem~1]{Doney:On_the_exact_asymptotic_behaviour_of_the_distribution_of_ladder_epochs}, it holds that
\begin{equation}\label{eq:Doney_temps_records}
\Pr{\tau_1 \ge n} \enskip\mathop{\sim}^{}_{n \to \infty}\enskip C \cdot n^{-\frac{\alpha-1}{\alpha}} L^\ast(n),
\end{equation}
with a constant $C > 0$ which shall not be important here. By a Tauberian theorem, see e.g.~\cite[Chapter~XVII, Theorem~5.5]{Feller:An_introduction_to_probability_theory_and_its_applications_Volume_2} it follows that
\begin{align*}
1-\Es{\e^{-\lambda \tau_1}}
&= (1-\e^{-\lambda}) \sum_{n \ge 0} \e^{-\lambda n} \Pr{\tau_1 > n}
\nonumber
\\
&\enskip\mathop{\sim}^{}_{\lambda \downarrow 0}\enskip
C \cdot \Gamma(1/\alpha) \cdot (1-\e^{-\lambda})^{\frac{\alpha-1}{\alpha}} L^\ast\left((1-\e^{-\lambda})^{-1}\right)
\nonumber
\\
&\enskip\mathop{\sim}^{}_{\lambda \downarrow 0}\enskip
C_\alpha \cdot \lambda^{\frac{\alpha-1}{\alpha}} \ell(\lambda^{-1}),
\end{align*}
for some constant $C_\alpha > 0$, where we recall that $\ell$ is a slowly varying function at infinity such that $B_n = n^{1/\alpha} \ell(n)
$, so, taking $\lambda = 1/N$ with $N \in \N$, we obtain
\begin{equation}\label{eq:Tauberien}
1-\Es{\e^{-\tau_1/N}}
\enskip\mathop{\sim}^{}_{N \to \infty}\enskip
C_\alpha \cdot N^{-1} \cdot B_N.
\end{equation}

\subsubsection{{\L}ukasiewicz paths and random walks}
\label{sec:Luka_marche}

Recall that $S = (S(i))_{i \ge 0}$ denotes a random walk started from $0$ with steps $(X_i)_{i \ge 1}$ given by i.i.d. random variables with law $\P(X_1 = k) = \mu(k+1)$ for every $k \ge -1$. Let $(X_n(i))_{1 \le i \le n+1}$ have the law of $(X_i)_{1 \le i \le n+1}$ conditioned to satisfy $X_1 + \dots + X_{n+2} = -1$ and let $S_n = (S_n(i))_{0 \le i \le n+2}$ be the associated path. For every $1 \le j \le n+2$, put
\[X^{(j)}_n(k) = X_n(k+j \text{ mod } n+2),
\qquad
1 \le k \le n+2.\]
We say that $X^{(j)}_n$ is the $j$-th cyclic shift of $X_n$. Obviously, for every $1 \le j \le n+2$, we have $X^{(j)}_{n,1} + \dots + X^{(j)}_{n,n+2} = -1$, but it turns out there is a unique $j$ such that $X^{(j)}_{n,1} + \dots + X^{(j)}_{n,k} \ge 0$ for every $1 \le k \le n$. This index is the least time at which the path $S_n$ achieves its minimum overall value:
\begin{equation}\label{eq:argim_pont_echangeable}
j = \inf\left\{1 \le k \le n+2 : S_n(k) = \inf_{1 \le i \le n+2} S_n(i)\right\}.
\end{equation}
Moreover, it is a standard fact that this time $j$ has the uniform distribution on $\{1, \dots, n+2\}$ and furthermore $X^{\ast}_n = X^{(j)}_n$ has the same law as the increments of the {\L}ukasiewicz path $W_n$ of the tree $T_n$ and it is independent of $j$. See e.g.~\cite[Chapter~6.1]{Pitman:Combinatorial_stochastic_processes} for details.

We see that cyclicly shifting the path $W_n$ at a fixed time, we obtain a random walk bridge $S_n$. The latter is also invariant in law under time and space \emph{reversal}, so by combining these observations, we obtain the following property: let $(X_n(i))_{1 \le i \le n+2}$ be the increments of $S_n$ and for a given $1 \le i \le n+2$, let $\widehat{X}_n^{(i)}(k) = X_n(i+1-k)$ for $1 \le k \le i$ and $\widehat{X}_n^{(i)}(k) = X_n(n+i+2-k)$ for $i+1 \le k \le n+2$; let $\widehat{S}^{(i)}_n$ be the associated path started from $0$, then it has the same distribution as $S_n$.

Let us finally note that the bridge conditioning is not important: an argument based on the Markov property of $S$ applied at time $\lceil n/2\rceil$ and the \emph{local limit theorem} shows that there exists a constant $C > 0$ such that for every event $A_n$ depending only on the first $\lceil n/2\rceil$ steps of the path, we have
\[\Prc{A_n}{S(n) = -1} \le C\cdot \Pr{A_n},\]
see e.g.~\cite{Kortchemski:Sub_exponential_tail_bounds_for_conditioned_stable_Bienayme_Galton_Watson_trees}, near the end of the proof of Theorem~9 there.

\subsubsection{The height process as local times}
\label{sec:hauteur_records}

Let us list the vertices of $T_n$ in lexicographical order as $\varnothing = u_0 < u_1 < \dots < u_n$. It is well-known that the processes $H_n$ and $W_n$ are related as follows (see e.g. Le Gall \& Le Jan~\cite{Le_Gall-Le_Jan:Branching_processes_in_Levy_processes_the_exploration_process}): for every $0 \le j \le n$,
\[H_n(j) = \#\left\{k \in \{0, \dots, j-1\} : W_n(k) \le \inf_{[k+1, j]} W_n\right\}.\]
Indeed, for $k < j$, we have $W_n(k) \le \inf_{[k+1, j]} W_n$ if and only if $u_k$ is an ancestor of $u_j$; moreover, the inequality is an equality if and only if the last child of $u_k$ is also an ancestor of $u_j$. Fix $i < j$ and suppose that $u_i$ is not an ancestor of $u_j$ (this case is treated similarly); denote by $\overline{ij} < i$ the index of the last common ancestor of $u_i$ and $u_j$, and $j' \in (i, j]$ the index of the child of $u_{\overline{ij}}$ which is an ancestor of $u_j$. It follows from the preceding identity that the quantity $W_n(i) - \min_{i \le k \le j} W_n(k)$ counts the number of vertices branching-off of the ancestral line $\llbracket u_{\overline{ij}}, u_i\llbracket$ which lie between $u_i$ and $u_j$, i.e. all the vertices visited between time $i$ and $j$ whose parent belongs to $\llbracket u_{\overline{ij}}, u_i\llbracket$. Indeed, started from $i$, the path $W_n$ will take only values larger than or equal to $W_n(i)$ until it visits the last ancestor of $u_i$, in which case it takes value exactly $W_n(i)$. Then $W_n$ will decrease by one exactly at every time it visits a vertex whose parent belongs to $\llbracket u_{\overline{ij}}, u_i\llbracket$, until the last one which is $u_{j'}$. We conclude that
\[W_n(j') = \inf_{i \le k \le j} W_n(k),
\qquad\text{and}\qquad
H_n(j') = \inf_{i \le k \le j} H_n(k).\]
It follows that the length of the path $\llbracket u_{j'}, u_j\rrbracket$ is
\begin{align*}
H_n(j) - H_n(j') &= \#\left\{k \in \{j', \dots, j\} : W_n(k) = \min_{k \le l \le j} W_n(l)\right\}
\\
&= \#\left\{k \in \{i, \dots, j\} : W_n(k) = \min_{k \le l \le j} W_n(l)\right\}.
\end{align*}

We can now prove Lemma~\ref{lem:Holder_hauteur_GW_stables} appealing to the preceding subsections.

\subsection{Proof of Lemma~\ref{lem:Holder_hauteur_GW_stables}}
\label{sec:tension_holder_hauteur}

Fix $\gamma \in (0, (\alpha-1)/\alpha)$. We claim that there exists a sequence of events $(E_n)_{n \ge 1}$ whose probability tends to $1$ such that the following holds. There exists $c_1, c_2 > 0$ such that for every $n$ large enough, every $0 \le s \le t \le 1$, and every $x \ge 0$, we have
\begin{equation}\label{eq:branche_droite}
\Pr{|H_n(nt) - \inf_{r \in [s,t]} H_n(nr)| \ge x \frac{n}{B_n} |t-s|^\gamma} \le c_1 \e^{- c_2 x},
\end{equation}
and
\begin{equation}\label{eq:branche_gauche}
\Prc{|H_n(ns) - \inf_{r \in [s,t]} H_n(nr)| \ge x \frac{n}{B_n} |t-s|^\gamma}{E_n} \le c_1 \e^{- c_2 x}.
\end{equation}
This shows that under the conditional probability $\P(\,\cdot\mid E_n)$, the moments of $\frac{B_n}{n} \frac{|H_n(nt) - H_n(ns)|}{|t-s|^\gamma}$ are bounded uniformly in $n$ and $s,t \in [0,1]$, so Lemma~\ref{lem:Holder_hauteur_GW_stables}, first under $\P(\,\cdot\mid E_n)$, but then under the unconditioned law, follows from Kolmogorov's tightness criterion. Let us start by considering the right branch and prove~\eqref{eq:branche_droite}. Note that we may, and shall, restrict to times $0 \le s < t \le 1$ such that $t-s \le 1/2$ and both $ns$ and $nt$ are integers.

\begin{proof}[Proof of~\eqref{eq:branche_droite}]
According to the discussion closing Section~\ref{sec:hauteur_records}, our claim~\eqref{eq:branche_droite} reads as follows: for every pair $s < t$,
\begin{equation}\label{eq:branche_droite_marche_records}
\Pr{\#\left\{k \in \{ns, \dots, nt\} : W_n(k) = \min_{k \le l \le nt} W_n(l)\right\} \ge x \frac{n}{B_n} |t-s|^\gamma} \le c_1 \e^{- c_2 x}.
\end{equation}
Let us first consider the random walk bridge $S_n$ and prove that~\eqref{eq:branche_droite_marche_records} holds when $W_n$ is replaced by $S_n$. Note that we may, and shall, restrict to times such that $t-s \le 1/2$ and both $ns$ and $nt$ are integers. By shifting the path at time $nt$ and then taking its time and space reversal, this cardinal of the set in this probability has the same law as the number of weak records of $S_n$ up to time $n|t-s|$. Let $(\tau_n(i))_{i \ge 0}$ be the weak record times of $S_n$, we therefore aim at bounding the probability
\[\Pr{\tau_n\left(\left\lfloor x \frac{n}{B_n} |t-s|^\gamma\right\rfloor\right) \le n |t-s|}.\]
Since $n|t-s| \le n/2$, as explained in Section~\ref{sec:Luka_marche}, this probability is bounded by some constant $C > 0$ times
\[\Pr{\tau\left(\left\lfloor x \frac{n}{B_n} |t-s|^\gamma\right\rfloor\right) \le n |t-s|},\]
where $(\tau(i))_{i \ge 0}$ are the weak record times of the unconditioned walk $S$. Recall that $(\tau(i+1) - \tau(i))_{i \ge 0}$ are i.d.d. and let $\tau = \tau(1)$. The exponential Markov inequality shows that the preceding probability is bounded by
\[\e \cdot \Es{\exp\left(-\frac{\tau(\lfloor x \frac{n}{B_n} |t-s|^\gamma\rfloor)}{n |t-s|}\right)}
= \exp\left(1 + \left\lfloor x \frac{n}{B_n} |t-s|^\gamma\right\rfloor \ln\left(1 - \left(1 - \Es{\exp\left(-\frac{\tau}{n |t-s|}\right)}\right)\right)\right).\]
From~\eqref{eq:Tauberien}, we get that
\begin{align*}
\left\lfloor x \frac{n}{B_n} |t-s|^\gamma\right\rfloor \ln\left(1 - \left(1 - \Es{\exp\left(\frac{\tau}{n |t-s|}\right)}\right)\right)
&= \left\lfloor x \frac{n}{B_n} |t-s|^\gamma\right\rfloor \ln\left(1 - C_\alpha \frac{B_{n |t-s|}}{n |t-s|} (1+o(1))\right)
\\
&= - x \frac{n}{B_n} |t-s|^\gamma C_\alpha \frac{B_{n |t-s|}}{n |t-s|} (1+o(1))
\\
&=- C_\alpha x \frac{(n |t-s|)^{-1/\alpha} B_{n |t-s|}}{n^{-1/\alpha} B_n} |t-s|^{\gamma-1+\frac{1}{\alpha}} (1+o(1)),
\end{align*}
where the $o(1)$ does not depend on $s$ and $t$. Let $\varepsilon = 1 - \frac{1}{\alpha} - \gamma > 0$, since the sequence $(n^{-1/\alpha} B_n)_{n \ge 1}$ is slowly varying, the so-called Potter bound (see e.g.\cite[Lemma~4.2]{Bjornberg-Stefansson:Random_walk_on_random_infinite_looptrees} or \cite[Equation~9]{Kortchemski:Sub_exponential_tail_bounds_for_conditioned_stable_Bienayme_Galton_Watson_trees}) asserts that there exists a constant $c$, depending on $\varepsilon$ (and so on $\gamma$), such that for every $n$ large enough,
\[\frac{(n |t-s|)^{-1/\alpha} B_{n |t-s|}}{n^{-1/\alpha} B_n} \ge c \cdot |t-s|^\varepsilon.\]
We conclude that 
\[\Pr{\#\left\{k \in \{ns, \dots, nt\} : S_n(k) = \min_{k \le l \le nt} S_n(l)\right\} \ge x \frac{n}{B_n} |t-s|^\gamma}
\le C \cdot \exp\left(1- C_\alpha c x (1+o(1))\right),\]
for every pair $s < t$, which indeed corresponds to~\eqref{eq:branche_droite_marche_records} with $S_n$ instead of $W_n$.

\begin{figure}[!ht] \centering
\includegraphics[page=1, width=.45\linewidth]{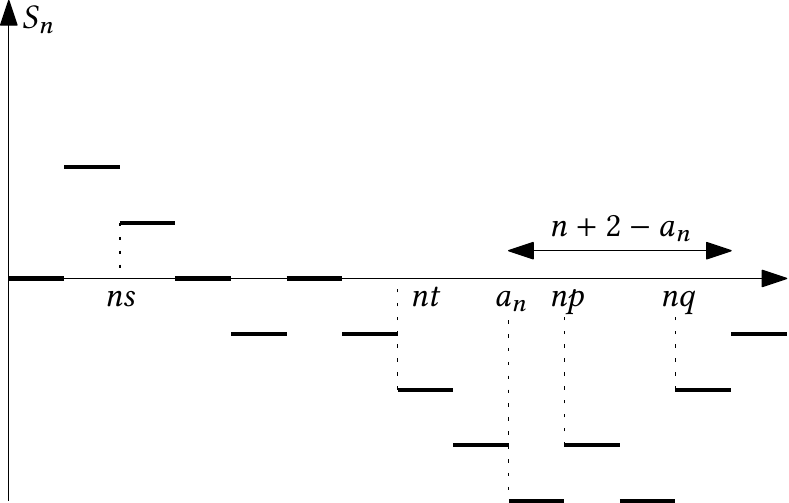}
\qquad
\includegraphics[page=2, width=.45\linewidth]{shift_excursion}
\caption{A bridge $S_n$ and its shifted excursion $W_n$; the times $s, t$ fall into the first case, whereas $p,q$ fall into the second case and $s,p$ into the third case.}
\label{fig:shift_excursion}
\end{figure}

We next prove~\eqref{eq:branche_droite_marche_records} by relating $W_n$ and $S_n$, as depicted in Figure~\ref{fig:shift_excursion}. Recall that these paths have length $n+2$.
Let us denote by $a_n$ the time $j$ in~\eqref{eq:argim_pont_echangeable} so the path $S_n$ shifted at time $a_n$ has the law of $W_n$. Fix two times $s < t$ such that $ns$ and $nt$ are integers and denote by $s'$ and $t'$ their respective image after the shift. We distinguish three cases:
\begin{enumerate}
\item Either $ns < nt \le a_n$, in which case $ns' = ns + (n+2-a_n) < nt + (n+2-a_n) = nt'$;
\item Either $a_n \le ns < nt$, in which case $ns' = ns - a_n < nt - a_n = nt'$; 
\item Or $ns < a_n < nt$, in which case $nt' = nt - a_n < ns + (n+2-a_n) = ns'$.
\end{enumerate}
In the first two cases, the parts of the two paths $(S_n(k))_{ns \le k \le nt}$ and $(W_n(k))_{ns' \le k \le nt'}$ are identical, and $t'-s' = t-s$ so, according to~\eqref{eq:branche_droite}, we have
\[\Pr{\#\left\{k \in \{ns', \dots, nt'\} : W_n(k) = \min_{k \le l \le nt'} W_n(l)\right\} \ge x \frac{n}{B_n} |t'-s'|^\gamma} \le c_1 \e^{- c_2 x}.\]
In the third case above, we have to be a little more careful; by cutting $W_n$ at time $n+2-a_n$ (which corresponds to $n+2$ for $S_n$), we observe that
\begin{align*}
&\#\left\{k \in \{nt', \dots, ns'\} : W_n(k) = \min_{k \le l \le ns'} W_n(l)\right\}
\\
&\le \#\left\{k \in \{nt', \dots, n+2-a_n\} : W_n(k) = \min_{k \le l \le n+2-a_n} W_n(l)\right\}
\\
&\qquad+ \#\left\{k \in \{n+2-a_n, \dots, ns'\} : W_n(k) = \min_{k \le l \le ns'} W_n(l)\right\}
\\
&= \#\left\{k \in \{nt, \dots, n+2\} : S_n(k) = \min_{k \le l \le n+2} S_n(l)\right\}
+ \#\left\{k \in \{0, \dots, ns\} : S_n(k) = \min_{k \le l \le ns} S_n(l)\right\}.
\end{align*}
A union bound then yields
\begin{align*}
&\Pr{\#\left\{k \in \{nt', \dots, ns'\} : W_n(k) = \min_{k \le l \le nt'} W_n(l)\right\} \ge x \frac{n}{B_n} |t'-s'|^\gamma}
\\
&\le\Pr{\#\left\{k \in \{nt, \dots, n+2\} : S_n(k) = \min_{k \le l \le n+2} S_n(l)\right\} \ge \frac{x}{2} \frac{n}{B_n} |t'-s'|^\gamma}
\\
&\qquad+ \Pr{\#\left\{k \in \{0, \dots, ns\} : S_n(k) = \min_{k \le l \le ns} S_n(l)\right\} \ge \frac{x}{2} \frac{n}{B_n} |t'-s'|^\gamma}
\\
&\le\Pr{\#\left\{k \in \{nt, \dots, n+2\} : S_n(k) = \min_{k \le l \le n+2} S_n(l)\right\} \ge \frac{x}{2} \frac{n}{B_n} |1-t|^\gamma}
\\
&\qquad+ \Pr{\#\left\{k \in \{0, \dots, ns\} : S_n(k) = \min_{k \le l \le ns} S_n(l)\right\} \ge \frac{x}{2} \frac{n}{B_n} |s|^\gamma}
\\
&\le c_1 \e^{- c_2 x},
\end{align*}
which concludes the proof of~\eqref{eq:branche_droite_marche_records}.
\end{proof}

The idea to control the left branch $|H_n(ns) - \inf_{r \in [s,t]} H_n(nr)|$ is to consider the `mirror tree' obtained from $T_n$ by flipping the order of the children of every vertex. There is one subtlety though, let us explain how to make this argument rigorous, with the help of Figure~\ref{fig:miroir}. Put $i=ns$ and $j = nt$. Let us denote by $\widetilde{T}_n$ the image of $T_n$ by the following two operations: first exchange the subtrees of the progeny of the $i$-th and the $j$-th vertices of $T_n$ and then take the mirror image of the whole tree, the resulting tree is $\widetilde{T}_n$. Observe that $T_n$ and $\widetilde{T}_n$ have the same law. Let $\tilde{i} > \tilde{j}$ be the indices such that the $\tilde{i}$-th and the $\tilde{j}$-th vertices of $\widetilde{T}_n$ correspond to the $i$-th and the $j$-th vertices of $T_n$ respectively. Then between times $i$ and $j$, in $T_n$, the {\L}ukasiewicz path $W_n$ visits all the progeny of the $i$-th vertex, then all the vertices that lie strictly between the two ancestral lines between the $i$-th and $j$-th vertices and their last common ancestor, and also all the vertices on this ancestral line leading to $j$. Similarly, between times $\tilde{j}$ and $\tilde{i}$, in $\widetilde{T}_n$, the {\L}ukasiewicz path $\widetilde{W}_n$ visits all the progeny of the $\tilde{j}$-th vertex, which is the same as that of the $i$-th vertex of $T_n$, then all the vertices that lie strictly between the two ancestral lines between the $\tilde{j}$-th and $\tilde{i}$-th vertices and their last common ancestor, which again are the same as in $T_n$, and also all the vertices on this ancestral line leading to $\tilde{i}$. So the two {\L}ukasiewicz paths visit the same vertices, except that $W_n$ visits the ancestors of the $j$-th vertex of $T_n$ and not those of its $i$-th vertex, whereas $\widetilde{W}_n$ visits the ancestors of the $i$-th vertex of $T_n$ and not those of its $j$-th vertex. In principle, the lexicographical distance $|\tilde{j} - \tilde{i}|$ may thus be much larger than $|i-j|$ so we cannot directly apply the bound~\eqref{eq:branche_droite} to $\widetilde{W}_n$ (note that it could also be smaller, but this is not an issue for us, it actually helps). The following lemma shows that this difference is indeed not important.

\begin{figure}[!ht] \centering
\includegraphics[page=1, width=.35\linewidth]{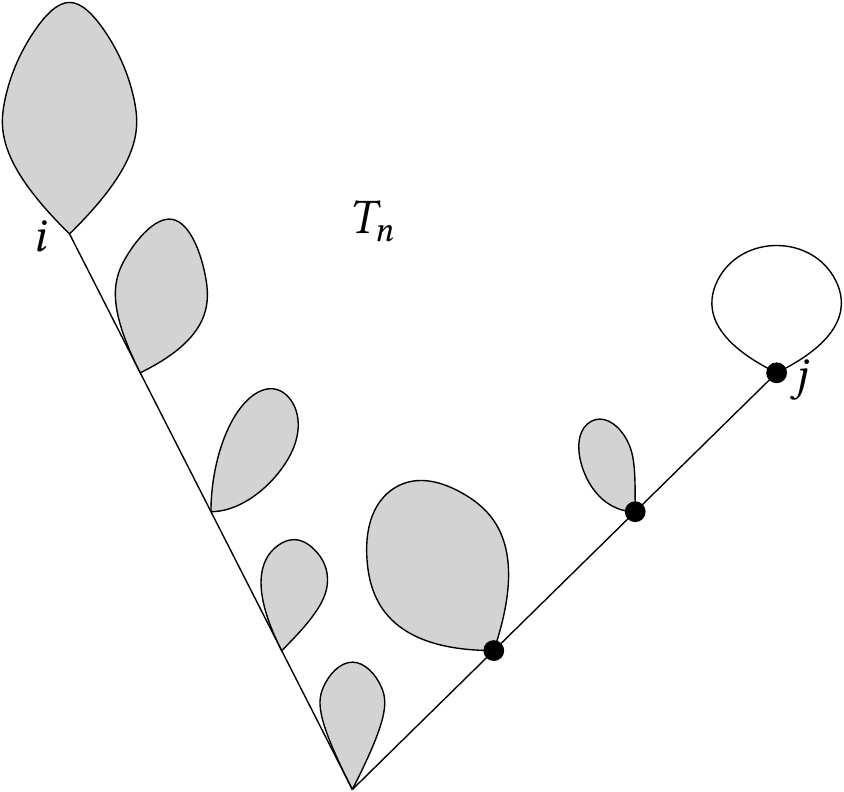}
\qquad
\includegraphics[page=2, width=.35\linewidth]{miroir}
\caption{On the left: a portion of the tree $T_n$ and two vertices $u_i$ and $u_j$; on the right: the `mirror' images $\widetilde{T}_n$, $\tilde{i}$ and $\tilde{j}$. The vertices visited by $W_n$ (resp. $\widetilde{W}_n$) between time $i$ and $j$ (resp. $\tilde{j}$ and $\tilde{i}$) are those black dots on the right branch as well as all the vertices strictly inside the grey trees.}
\label{fig:miroir}
\end{figure}

Recall that for a vertex $v$ of a tree $T$ different from its root, we denote by $pr(v)$ its parent and by $k_{pr(v)}$ the number of children of the latter; denote further by $\chi_v$ the relative position of $v$ among the children of $pr(v)$: formally, the index $\chi_v \in \{1, \dots, k_{pr(v)}\}$ satisfies $v=pr(v)\chi_v$.

\begin{lem}\label{lem:longueur_branches}
Let $C = \frac{10}{\mu(0)^2}$, then the probability of the event
\[\left\{\frac{\#\{w \in \mathopen{\rrbracket} u, v \rrbracket : \chi_w = 1\}}{\# \mathopen{\rrbracket} u, v \rrbracket} \le 1-\frac{\mu(0)}{2} \text{ for every } u, v \in T_n \text{ such that } u \in \llbracket \varnothing, v \llbracket \text{ and } \# \mathopen{\rrbracket} u, v \rrbracket > C \ln n\right\}\]
tends to $1$ as $n \to \infty$.
\end{lem}

We can now finish the proof of Lemma~\ref{lem:Holder_hauteur_GW_stables}.

\begin{proof}[Proof of~\eqref{eq:branche_gauche}]
From the preceding lemma, we deduce that there exists some $p \in (0,1)$ such that with high probability, on all ancestral paths in $T_n$ of length at least logarithmic, there is a proportion at least $p$ of individuals which are not the first child of their parent; symmetrically, there is the same proportion of individuals which are not the last child of their parent. Consequently, the length of such a path, multiplied by $p$, is bounded below by the number of vertices whose parent belongs to this path, and which themselves lie strictly to its right. With the notation of the discussion preceding the lemma, on the event described in this lemma, the lexicographical distance in $\widetilde{T}_n$ between the images of the $i$-th and $j$-th vertex of $T_n$ is
\begin{align*}
|\tilde{j} - \tilde{i}| 
&= |i-j| - |H_n(j) - \inf_{i \le k \le j} H_n(k)| + |H_n(i) - \inf_{i \le k \le j} H_n(k)|
\\
&\le |i-j| + p^{-1} |i-j|,
\end{align*}
where the second (very rough) bound holds only if $|H_n(i) - \inf_{i \le k \le j} H_n(k)| > C \ln n$, with $C$ as in Lemma~\ref{lem:longueur_branches}. Note that~\eqref{eq:branche_gauche} is trivial otherwise since $|t-s| \ge \frac{1}{n}$ as we restricted to integer times, so $x \frac{n}{B_n} |t-s|^\gamma \ge x \frac{n^{1-\gamma}}{B_n}$ which tends to infinity like a power of $n$. We then conclude from the bound~\eqref{eq:branche_droite} applied to the `mirror' {\L}ukasiewicz path $\widetilde{W}_n$.
\end{proof}

It remains to prove Lemma~\ref{lem:longueur_branches}. A similar statement was proved in~\cite[Corollary~3]{Marzouk:Scaling_limits_of_random_bipartite_planar_maps_with_a_prescribed_degree_sequence} in the context of trees `with a prescribed degree sequence'. The argument may be extended to our present case but we chose to modify it in order to directly use the existing references on Bienaymé--Galton--Watson trees.

\begin{proof}[Proof of Lemma~\ref{lem:longueur_branches}]
Fix $\varepsilon > 0$ and let $T$ be an unconditioned Bienaymé--Galton--Watson tree with offspring distribution $\mu$. Define the set
\[A_T = \left\{u, v \in T: u \in \llbracket \varnothing, v \llbracket \text{ and } \# \mathopen{\rrbracket} u, v \rrbracket > C \ln n \text{ and } \frac{\#\{w \in \mathopen{\rrbracket} u, v \rrbracket : \chi_w = 1\}}{\# \mathopen{\rrbracket} u, v \rrbracket} > 1-\frac{\mu(0)}{2}\right\}.\]
Note that the maximal height of a tree cannot exceed its total size. Then by sampling $v$ uniformly at random in $T$, we obtain
\begin{align}\label{eq:moyenne_sommets_hauteur_h}
\Pr{\exists (u, v) \in A_T \text{ and } \#T = n+1}
&= \frac{1}{n+1} \Es{\sum_{h =1}^{n+1} \sum_{\substack{v \in T \\ |v| = h}} \ind{\exists u \in T \text{ such that } (u,v) \in A_T} \ind{\#T = n+1}}
\nonumber
\\
&\le \frac{1}{n+1} \sum_{h =1}^{n+1} \Es{\sum_{\substack{v \in T \\ |v| = h}} \ind{\exists u \in T \text{ such that } (u,v) \in A_T}}.
\end{align}
We then use a spinal decomposition due to Duquesne~\cite[Equation~24]{Duquesne:An_elementary_proof_of_Hawkes_s_conjecture_on_Galton_Watson_trees} which results in an absolute continuity relation between the tree $T$ and the tree $T_\infty$ `conditioned to survive', which is the infinite tree which arise as the \emph{local limit} of $T_n$. It was introduced by Kesten~\cite{Kesten:Subdiffusive_behavior_of_random_walk_on_a_random_cluster} and the most general results on such convergences are due to Abraham \& Delmas~\cite{Abraham-Delmas:Local_limits_of_conditioned_Galton_Watson_trees_the_infinite_spine_case}. The tree $T_\infty$ contains a unique infinite simple path called the \emph{spine}, starting from the root, and the vertices which belong to this spine reproduce according to the \emph{size-biased} law $(k \mu(k))_{k \ge 1}$, whereas the other vertices reproduce according to $\mu$, and all the vertices reproduce independently. For a tree $\tau$ and a vertex $v \in \tau$, let $\theta_v(\tau)$ be the subtree consisting of $v$ and all its progeny, and let $\mathsf{Cut}_v(\tau) = \{v\} \cup (\tau \setminus \theta_v(\tau))$ be its complement (note that $v$ belongs to the two parts). Then for every non-negative measurable functions $G_1, G_2$, for every $h \ge 0$, we have
\[\Es{\sum_{\substack{v \in T \\ |v| = h}} G_1(\mathsf{Cut}_v(T), v) \cdot G_2(\theta_v(T))}
= \Es{G_1(\mathsf{Cut}_{v_h^\ast}(T_\infty), v_h^\ast)} \cdot \Es{G_2(T)},\]
where $v_h^\ast$ is the only vertex on the spine of $T_\infty$ at height $h$. Then the expectation in~\eqref{eq:moyenne_sommets_hauteur_h} equals
\begin{align*}
&\Pr{\exists u \in T_\infty: u \in \llbracket \varnothing, v_h^\ast \llbracket \text{ and } \# \mathopen{\rrbracket} u, v_h^\ast \rrbracket > C \ln n \text{ and } \frac{\#\{w \in \mathopen{\rrbracket} u, v_h^\ast \rrbracket : \chi_w = 1\}}{\# \mathopen{\rrbracket} u, v_h^\ast \rrbracket} > 1-\frac{\mu(0)}{2}}
\\
&\le \sum_{k = C \ln n}^{h-1} \Pr{\frac{\#\{w \in \mathopen{\rrbracket} v_{h-k}^\ast, v_h^\ast \rrbracket : \chi_w = 1\}}{k} > 1-\frac{\mu(0)}{2}}.
\end{align*}
Now recall that on the spine, the vertices reproduce according to the size-biased law $(i \mu(i))_{i \ge 1}$, and furthermore, conditional on the number of children of its parent, the position of a vertex amongst its sibling is uniformly chosen, this means that for every vertex $w$ on the spine, the probability $\P(\chi_w = 1)$ equals $\sum_{i \ge 1} (i \mu(i))/i = 1-\mu(0)$ and these events are independent. Therefore, if $\mathsf{Bin}(N,p)$ denotes a random variable with the binomial law with parameters $N$ and $p$, then the preceding probability is bounded by
\[\sum_{k = C \ln n}^{h-1} \Pr{k^{-1} \mathsf{Bin}(k,1-\mu(0)) > 1-\frac{\mu(0)}{2}}
\le \sum_{k = C \ln n}^{h-1} \e^{-k \mu(0)^2/2},\]
where we have used the celebrated Chernoff bound. Putting things together, we obtain the bound
\begin{align*}
\Prc{\exists (u, v) \in A_T}{\#T = n+1}
&\le \frac{1}{(n+1) \P(\#T = n+1)} \sum_{h =1}^{n+1} \sum_{k = C \ln n}^{h-1} \e^{-k \mu(0)^2/2}
\\
&\le \frac{1}{(n+1) \P(\#T = n+1)} \sum_{h =1}^{n+1} h \cdot \e^{-C \ln n \mu(0)^2/2}.
\end{align*}
It is well-known that $n B_n \P(\#T = n+1) \to p_1(0)$ as $n \to \infty$, where $p_1$ is the density of the stable random variable $X^{(\alpha)}$ from the introduction; this follows e.g. from the fact that the event $\P(\#T = n+1)$ is the probability that the random walk $S$ first hits $-1$ at time $n+2$, which equals by cyclic shift $(n+1)^{-1}$ times the probability that $S_{n+2} = -1$ and the asymptotic behaviour of this probability is dictated by the local limit theorem, see e.g.~\cite[Lemma~1]{Kortchemski:A_simple_proof_of_Duquesne_s_theorem_on_contour_processes_of_conditioned_Galton_Watson_trees}. We conclude that for $n$ large enough
\begin{align*}
\Prc{\exists (u, v) \in A_T}{\#T = n+1}
&\le \frac{n^2 B_n}{2p_1(0)} \e^{-C \ln n \mu(0)^2/2} (1+o(1))
\\
&\le n^3 \e^{-C \ln n \mu(0)^2/2},
\end{align*}
which converges to $0$ from our choice of $C$.
\end{proof}

\section{Convergence of snakes}
\label{sec:serpents}

We prove in this section the results presented in the introduction when we add to $T_n$ spatial positions given by i.i.d. increments with law $Y$. Recall that we concentrate only on the joint convergence of $H_n$ and $\Hsp_n$.

\subsection{Proof of Theorem~\ref{thm:convergence_serpent_iid} for centred snakes}
Let us first focus on the case $\E[Y]=0$; we aim at showing the convergence in distribution in $\mathscr{C}([0,1], \R^2)$
\[\left(\frac{B_n}{n} H_n(n t), \left(\frac{B_n}{n \Sigma^2}\right)^{1/2} \Hsp_n(nt)\right)_{t \in [0,1]} \cvloi (\Hexc_t, \Snake_t)_{t \in [0,1]}.\]
where $\Sigma^2 \coloneqq \E[Y^2] \in (0, \infty)$; this convergence in the sense of finite-dimensional marginals follows easily from~\eqref{eq:cv_GW_Duquesne} appealing e.g. to Skorohod's representation theorem and Donsker's invariance principle applied to finitely many branches. We thus only focus on the tightness of the rescaled process $(B_n/n)^{1/2} \Hsp_n(n\cdot)$. The idea is to apply Kolmogorov's criterion but our assumption does not give us sufficiently large moments. We therefore adapt the argument from~\cite{Janson-Marckert:Convergence_of_discrete_snakes} and treat separately the large and small values of $Y$'s: the large ones are too rare to contribute much and the small ones now have sufficiently large moments. The proof takes five steps.

\subsubsection{Necessity of the assumption}\label{sec:thm_principal_step1}
Suppose first that the assumption $\P(|Y| \ge (n/B_n)^{1/2}) = o(n^{-1})$ does not hold. Then there exists $\delta > 0$ such that for infinitely many indices $n \in \N$, we have $\P(|Y| \ge (n/B_n)^{1/2}) \ge \delta n^{-1}$. Let us implicitly restrict ourselves to such indices; let $(Y_i)_{i \ge 1}$ be i.d.d. copies of $Y$, independent of $T_n$, then the conditional probability given $T_n$ that there exists an internal vertex $u$ such that its first child satisfies $|Y_{u1}| \ge (n/B_n)^{1/2}$ equals
\[\Prc{\bigcup_{i = 1}^{n+1-\lambda(T_n)} \{|Y_i| \ge (n/B_n)^{1/2}\}}{T_n}
= 1 - \Prc{\bigcap_{i = 1}^{n+1-\lambda(T_n)} \{|Y_i| < (n/B_n)^{1/2}\}}{T_n}
\ge 1 - \left(1-\frac{\delta}{n}\right)^{n+1-\lambda(T_n)},\]
where $\lambda(T_n)$ denotes the number of leaves of the tree. Since $\limsup_{n \to \infty} \lambda(T_n)/n < 1$ with high probability, and indeed $\lambda(T_n)/n$ converges to $\mu(0)$, see e.g.~\cite[Lemma~2.5]{Kortchemski:Invariance_principles_for_Galton_Watson_trees_conditioned_on_the_number_of_leaves}, the right-most term is bounded away from $0$ uniformly in $n$. We conclude that with a probability bounded away from $0$, for infinitely many indices $n \in \N$, there exists $0 \le i < n$ such that $(n/B_n)^{1/2} |\Hsp_n(i+1) - \Hsp_n(i)| \ge 1$ so the sequence of continuous processes $((n/B_n)^{1/2} \Hsp_n(n\cdot))_{n \ge 1}$ cannot be tight.

\subsubsection{A cut-off argument}\label{sec:thm_principal_step2}
We assume for the rest of the proof that $\P(|Y| \ge (n/B_n)^{1/2}) = o(n^{-1})$. Recall that for every $\delta > 0$, we have $n^{\frac{1}{\alpha}-\delta} \ll B_n \ll n^{\frac{1}{\alpha}+\delta}$ so this assumption implies $\P(|Y| \ge y) = o(y^{-\frac{2\alpha}{\alpha(1+\delta)-1}})$. Set $b_n = (n^2/B_n)^{\frac{\alpha-1}{4\alpha} + \varepsilon}$ for some $\varepsilon > 0$; we shall tune $\varepsilon$ and $\delta$ small. The idea is to take into account separately the large increments. For every vertex $u \in T_n$, let $Y_u' = Y_u \ind{|Y_u| \le b_n}$ and $Y_u'' = Y_u \ind{|Y_u| > b_n}$, define then $\Hsp_n\vphantom{H}'$ and $\Hsp_n\vphantom{H}''$ the spatial processes in which the increments $Y_u$ are replaced by $Y_u'$ and $Y_u''$ respectively, so $\Hsp_n = \Hsp_n\vphantom{H}' + \Hsp_n\vphantom{H}''$.

\subsubsection{Contribution of the large jumps}\label{sec:thm_principal_step3}
Let $E_n$ be the event that $T_n$ contains two vertices, say $u$ and $v$, such that $u$ is an ancestor of $v$ and both $|Y_u| > b_n$ and $|Y_v| > b_n$. Then $\P(E_n \mid T_n) \le \Lambda(T_n) \P(|Y| > b_n)^2$, where $\Lambda(T_n) = \sum_{u \in T_n} |u|$ is called the \emph{total path length} of $T_n$. It is a simple matter to prove the following well-known integral representation: if $C_n$ denotes the contour process of $T_n$, then
\[\Lambda(T_n) = \frac{n}{2} + \frac{1}{2} \int_0^{2n} C_n(t)\d t = \frac{n^2}{B_n} \left(\frac{B_n}{2n} + \int_0^1 \frac{B_n}{n} C_n(2nt)\d t \right).\]
We deduce that if $T_n$ satisfies~\ref{eq:cv_GW_Duquesne}, then we have the convergence in distribution:
\begin{equation}\label{eq:convergence_total_path_length}
\frac{B_n}{n^2} \Lambda(T_n) \cvloi \int_0^1 \Hexc_t \d t.
\end{equation}
We then write for every $K > 0$,
\[\limsup_{n \to \infty} \Pr{E_n}
\le \limsup_{n \to \infty} \Pr{\Lambda(T_n) > K n^2/B_n} + K \limsup_{n \to \infty} \frac{n^2}{B_n} \P(|Y| > b_n)^2.\]
The first term on the right tends $0$ when $K \to \infty$, and as for the second term, from our choice of $b_n$, we have for every $\delta > 0$,
\[\frac{n^2}{B_n} \P(|Y| > b_n)^2
\ll \left(\frac{n^2}{B_n}\right)^{1-(\frac{\alpha-1}{4\alpha} + \varepsilon) (\frac{4\alpha}{\alpha(1+\delta)-1})},\]
and the exponent is negative for $\delta$ sufficiently small. Now on the event $E_n^c$, there is at most one edge on each branch along which the spatial displacement is in absolute value larger than $b_n$, therefore $\max_{0 \le i \le n} |\Hsp_n\vphantom{H}''(i)|$ simply equals $\max_{u \in T_n} |Y_u''|$ and so for every $\delta > 0$,
\begin{align*}
\Pr{\left\{\max_{0 \le t \le 1} |\Hsp_n\vphantom{H}''(2nt)| > \delta (n/B_n)^{1/2}\right\} \cap E_n^c}
&\le \Pr{\max_{u \in T_n} |Y_u| > \delta (n/B_n)^{1/2}}
\\
&\le n \Pr{|Y| > (n/B_n)^{1/2}},
\end{align*}
which converges to $0$ as $n \to \infty$. Thus $(B_n/n)^{1/2} \Hsp_n\vphantom{H}''(n\cdot)$ converges to $0$ so it only remains to prove that $(B_n/n)^{1/2} \Hsp_n\vphantom{H}'(n\cdot)$ is tight.

\subsubsection{Average contribution of small jumps}\label{sec:thm_principal_step4}
The process $(B_n/n)^{1/2} \Hsp_n\vphantom{H}'(n\cdot)$ is simpler to analyse than $(B_n/n)^{1/2} \Hsp_n(n\cdot)$ since all the increments are bounded in absolute value by $(n/B_n)^{-\varepsilon}$. Note nonetheless that $\Hsp_n\vphantom{H}'$ is non centred in general, we next prove that its conditional expectation given $T_n$ is negligible. Let $m_n = \E[Y'] = -\E[Y'']$ and observe that $\E[\Hsp_n\vphantom{H}'(n \cdot) \mid T_n] = m_n H_n(n \cdot)$. Recall that $(B_n/n) H_n(n\cdot)$ converges in distribution (to $\Hexc$); from the tail behaviour of $Y$ we get:
\[|m_n| \le \E[|Y''|]
\le b_n \Pr{|Y| > b_n} + \int_{b_n}^\infty \Pr{|Y| > y} \d y
= O\left(b_n^{1-\frac{2\alpha}{\alpha(1+\delta)-1}}\right).\]
Note that for $\varepsilon$ and $\delta$ sufficiently small, we have
\[\left(\frac{\alpha-1}{4\alpha} + \varepsilon\right) \left(1-\frac{2\alpha}{\alpha(1+\delta)-1}\right)
> - \frac{\alpha+1}{4\alpha},\]
and this bound gets tighter as $\varepsilon$ and $\delta$ get closer to $0$. Then, since
\[\frac{\alpha+1}{4\alpha} < \frac{1}{2} < \frac{\alpha+1}{2\alpha},\]
we conclude that for $\varepsilon$ and $\delta$ sufficiently small,
\[\left(\frac{n}{B_n}\right)^{1/2} \left(b_n^{1-\frac{2\alpha}{\alpha(1+\delta)-1}}\right)
= n^{2 (\frac{\alpha-1}{4\alpha} + \varepsilon)(1-\frac{2\alpha}{\alpha(1+\delta)-1}) + \frac{1}{2}} \cdot B_n^{- (\frac{\alpha-1}{4\alpha} + \varepsilon)(1-\frac{2\alpha}{\alpha(1+\delta)-1}) - \frac{1}{2}}\]
converges to $0$ as $n \to \infty$ since both exponents are negative. Therefore the process $(B_n/n)^{1/2} \E[\Hsp_n\vphantom{H}'(n\cdot)\mid T_n] = (n/B_n)^{1/2} m_n \cdot (B_n/n) H_n(n\cdot)$ converges in probability to $0$ and we focus for the rest of the proof on the centred process $\Hspc_n\vphantom{H}'(n\cdot) = \Hsp_n\vphantom{H}'(n\cdot) - \E[\Hsp_n\vphantom{H}'(n\cdot)\mid T_n]$.

\subsubsection{Re-centred small jumps are tight}\label{sec:thm_principal_step5}
It only remains to prove that $(B_n/n)^{1/2} \Hspc_n\vphantom{H}'(n\cdot)$ is tight. Fix $\gamma \in (0,(\alpha-1)/\alpha)$, let $\eta > 0$ arbitrary and let us fix $C > 0$ such that, according to Lemma~\ref{lem:Holder_hauteur_GW_stables}, for every $n$ large enough,
\[\Pr{\sup_{0 \le s \ne t \le 1} \frac{B_n \cdot |H_n(nt) - H_n(ns)|}{n \cdot |t-s|^\gamma} \le C} \ge 1-\eta.\]
We shall denote by $A_n$ the event in the preceding probability. Our aim is to apply Kolmogorov's tightness criterion to $(B_n/n)^{1/2} \Hspc_n\vphantom{H}'(n\cdot)$ on the event $A_n$. Let us enumerate the vertices of $T_n$ in lexicographical order as $u_0 < u_1 < \dots < u_n$. Fix $0 \le s < t \le 1$ such that $ns$ and $nt$ are both integers. Then $\Hspc_n\vphantom{H}'(nt) - \Hspc_n\vphantom{H}'(ns)$ is the sum of $\#\mathopen{\llbracket} u_{ns}, u_{nt}\mathclose{\llbracket}$ i.i.d. random variables distributed as $\mathsf{Y}' = Y' - \E[Y']$. Let $r \in [s,t]$ be as follows: set $r=s$ if $u_{ns}$ is an ancestor of $u_{nt}$; otherwise, $nr$ is an integer and $u_{nr}$ is the ancestor of $u_{nt}$ whose parent is the last common ancestor of $u_{ns}$ and $u_{nt}$. In this way, $r$ satisfies $H_n(nr) = \inf_{[s, t]} H_n(n \cdot)$ and it holds that
\[\#\mathopen{\llbracket} u_{ns}, u_{nt}\mathclose{\llbracket}
\le 2+ H_n(ns) + H_n(nt) - 2H_n(nr).\]
On the event $A_n$, the right-hand side is bounded by
\[C \frac{n}{B_n} (|t-r|^\gamma + |r-s|^\gamma)
\le 2C \frac{n}{B_n} |t-s|^\gamma.\]
Fix any $q \ge 2$ and let us write $C_q$ for a constant which will vary from one line to the other, and which depends on $q$ and the law of $Y$, but not on $s,t$ nor $n$. 

Note that $\E[|\mathsf{Y}'|^2] = \Var(Y') \le \E[|Y'|^2] \le \E[|Y|^2] < \infty$ and $|\mathsf{Y}'|^q \le 2^q(n^2/B_n)^{q(\frac{\alpha-1}{4\alpha} + \varepsilon)}$. Appealing to~\cite[Theorem~2.9]{Petrov:Limit_theorems_of_probability_theory} (sometimes called the Rosenthal inequality), we obtain
\begin{align*}
\Esc{\left(\frac{|\Hspc_n\vphantom{H}'(nt) - \Hspc_n\vphantom{H}'(ns)|}{(n/B_n)^{1/2}}\right)^q}{A_n}
&\le C_q \left(\frac{B_n}{n}\right)^{q/2} \left(\frac{n}{B_n} |t-s|^\gamma \Es{|\mathsf{Y}'|^q} + \left(\frac{n}{B_n} |t-s|^\gamma\right)^{q/2} \Es{|\mathsf{Y}'|^2}^{q/2}\right)
\\
&\le C_q \left(\left(\frac{B_n}{n}\right)^{q/2-1} \left(\frac{n^2}{B_n}\right)^{q(\frac{\alpha-1}{4\alpha} + \varepsilon)} |t-s|^\gamma + |t-s|^{q\gamma/2}\right).
\end{align*}
Recall that for every $\delta > 0$, we have $B_n \ll n^{\frac{1}{\alpha}+\delta}$ so
\[\frac{1}{\ln n} \ln\left(\left(\frac{B_n}{n}\right)^{q/2-1} \left(\frac{n^2}{B_n}\right)^{q(\frac{\alpha-1}{4\alpha} + \varepsilon)}\right)
\le \left(\frac{1}{\alpha}+\delta\right) \left(\frac{q}{2}-1 - q\left(\frac{\alpha-1}{4\alpha} + \varepsilon\right)\right)
+ 2q\left(\frac{\alpha-1}{4\alpha} + \varepsilon\right) - \left(\frac{q}{2}-1\right).\]
Taking $\varepsilon=\delta=0$, the right-hand side reads
\begin{align*}
\frac{1}{\alpha} \left(\frac{q}{2}-1 - q\frac{\alpha-1}{4\alpha}\right) + q\frac{\alpha-1}{2\alpha} - \left(\frac{q}{2}-1\right)
&=\frac{\alpha-1}{\alpha} \left(1 - \frac{q}{4\alpha}\right),
\end{align*}
which tends to $-\infty$ as $q \to \infty$. Now if $\varepsilon,\delta > 0$ are small, one obtains instead the exponent
\[\frac{\alpha-1}{\alpha} \left(1 - \frac{q}{4\alpha}\left[1 - \delta\left(\frac{4\alpha^2}{\alpha-1} + \frac{1}{2}\right) + 2\alpha \varepsilon \left(2\left(\frac{1}{\alpha}+\delta\right)-1\right)\right]\right) - \delta,\]
which still tends to $-\infty$ as $q \to \infty$. Notice also that $n^{-1} \le |s-t| \le 1$, so we may choose $q$ large enough so that
\[\sup_{n \to \infty} \Esc{\left(\frac{|\Hspc_n\vphantom{H}'(nt) - \Hspc_n\vphantom{H}'(ns)|}{(n/B_n)^{1/2}}\right)^q}{A_n}
\le C_q |t-s|^2.\]
This bound holds whenever $s, t \in [0, 1]$ are such that $ns$ and $nt$ are both integers. Since $\Hspc_n\vphantom{H}'$ is defined by linear interpolation between such times, then it also holds for every $s, t \in [0,1]$. The standard Kolmogorov criterion then implies the following bound for the H\"{o}lder norm of $\Hspc_n\vphantom{H}'$ for some $\theta >0$: for every $\delta > 0$, there exists $C > 0$ such that for every $n$ large enough,
\[\Prc{\sup_{0 \le s \ne t \le 1} \frac{|\Hspc_n(nt) - \Hspc_n(ns)|}{(n/B_n)^{1/2} |t-s|^\theta} \le C}{A_n} \ge 1-\delta.\]
Since $\P(A_n) \to 1$ as $n \to \infty$, the same result holds under the unconditional probability, which shows that the sequence $(B_n/n)^{1/2} \Hspc_n\vphantom{H}'(n\cdot)$ is indeed tight and the proof is complete.

\subsection{Proof of Theorem~\ref{thm:convergence_serpent_non_centre} for non-centred snakes}
We next assume that $\E[Y] = m \ne 0$ and prove Theorem~\ref{thm:convergence_serpent_non_centre}. The intuition behind the result is that the fluctuations are small and disappear after scaling, only the contribution of the expected displacement remains. Indeed, as in the preceding proof, we have $\frac{B_n}{n} \E[\Hsp_n(n \cdot) \mid T_n] = m \frac{B_n}{n} H_n(n \cdot)$ which converges to $m \cdot \Hexc$ so it is equivalent to consider the centred version of $Y$. For the rest of the proof, we thus assume that $\E[Y]=0$ and $\P(|Y| \ge n/B_n) = o(n^{-1})$, and we prove that the corresponding scaled spatial process $\frac{B_n}{n} \Hsp_n(n\cdot)$ converges to the null process.

The fact that our assumption is necessary for tightness of this process goes exactly as for Theorem~\ref{thm:convergence_serpent_iid}, in the first step: Now the tails of $Y$ are so that $\P(|Y| \ge y) = o(y^{-\frac{\alpha}{\alpha(1+\delta)-1}})$ for every $\delta > 0$ and we may proceed as previously, with the sequence $b_n = (n^2/B_n)^{\frac{\alpha-1}{2\alpha} + \varepsilon}$ instead: up to $\delta, \varepsilon$, both exponents in the tails of $Y$ and in $b_n$ are half what they were in the preceding section, so these changes compensate each other. Then the previous arguments apply \emph{mutatis mutandis}: we have
\[\lim_{n \to \infty} \frac{n^2}{B_n} \P(|Y| > b_n)^2 = 0
\qquad\text{and}\qquad
\lim_{n \to \infty} \left(\frac{n}{B_n}\right)^{1/2} b_n \Pr{|Y| > b_n}= 0,\]
so both processes $\frac{B_n}{n} \Hsp_n\vphantom{H}''(n\cdot)$ and $\frac{B_n}{n} \E[\Hsp_n\vphantom{H}'(n\cdot)\mid T_n]$ converge to the null process. Similarly, with the preceding notations, for $s,t \in [0,1]$, we have
\begin{align*}
\Esc{\left(\frac{|\Hspc_n\vphantom{H}'(nt) - \Hspc_n\vphantom{H}'(ns)|}{n/B_n}\right)^q}{A_n}
&\le C_q \left(\frac{B_n}{n}\right)^q \left(\frac{n}{B_n} |t-s|^\gamma \Es{|\mathsf{Y}'|^q} + \left(\frac{n}{B_n} |t-s|^\gamma\right)^{q/2} \Es{|\mathsf{Y}'|^2}^{q/2}\right)
\\
&\le C_q \left(\left(\frac{B_n}{n}\right)^{q-1} \left(\frac{n^2}{B_n}\right)^{q(\frac{\alpha-1}{2\alpha} + \varepsilon)} |t-s|^\gamma + \left(\frac{B_n}{n} \Es{|Y'|^2} |t-s|^\gamma\right)^{q/2}\right).
\end{align*}
The first term in the last line is controlled as previously: the factor $1/2$ in the exponent in $b_n$ compensate the fact that we now rescale by $n/B_n$ instead of $(n/B_n)^{1/2}$ and similar calculations as in the preceding section show that this first term is bounded by $|t-s|^\gamma$ times $n$ raised to a power which converges to $-\infty$ as $q \to \infty$. The only change compared to the proof of Theorem~\ref{thm:convergence_serpent_iid} is that we may not have $\E[|Y|^2] < \infty$. Still,
\[\Es{|Y'|^2}
= 2 \int_0^{b_n} y \Pr{|Y| > y} \d y
= O\left(\int_1^{b_n} y^{1-\frac{\alpha}{\alpha(1+\delta)-1}} \d y\right).\]
Note that if $\alpha < 2$, then for $\delta$ sufficiently small, the exponent is smaller than $-1$ so the integral converges. If $\alpha = 2$, then since $B_n$ is at least of order $n^{1/2}$ (and it is exactly of this order if and only if $\mu$ has finite variance), then we do not need any $\delta$: we have $\P(|Y| \ge y) = o(y^{-2})$ and so
\[\Es{|Y'|^2} = O\left(\int_1^{b_n} y^{-1} \d y\right) = O(\ln b_n) = O(\ln n).\]
In both cases, $\frac{B_n}{n} \E[|Y'|^2]$ is bounded above by $n^{-\eta}$ for some $\eta > 0$ and we may conclude as in the preceding proof that for $q$ large enough,
\[\sup_{n \to \infty} \Esc{\left(\frac{|\Hspc_n\vphantom{H}'(nt) - \Hspc_n\vphantom{H}'(ns)|}{n/B_n}\right)^q}{A_n}
\le C_q |t-s|^2,\]
and so the process $\frac{B_n}{n} \Hspc_n\vphantom{H}'(n\cdot)$ is tight. Moreover, the preceding bounds applied with $s=0$ and $t \in [0,1]$ fixed show that the one-dimensional marginals converge in distribution to $0$ so the whole process converges in distribution to the null process, which completes the proof.

\subsection{Application to the number of inversions}

Before discussing heavy-tailed snakes, let us apply Theorem~\ref{thm:convergence_serpent_iid} to prove Corollary~\ref{cor:inversion}, following the argument of Cai \emph{et al.}~\cite[Section 5]{Cai-Holmgren-Janson-Johansson-Skerman:Inversions_in_split_trees_and_conditional_Galton_Watson_trees}. 

First note that for a given tree $T$ with $n+1$ vertices listed $\varnothing = u_0 < u_1 < \dots < u_n$ in lexicographical order, we have
\[\E[I(T)] 
= \frac{1}{2} \sum_{0 \le i < j \le n} \ind{u_i \text{ is a ancestor of } u_j}
= \frac{1}{2} \sum_{u \in T} |u|
= \frac{1}{2} \Lambda(T),\]
where we recall the notation $\Lambda(T)$ for the total path length of $T$. Therefore the convergence of the conditional expectation of $I(T_n)$ follows from~\eqref{eq:convergence_total_path_length}. We focus on the fluctuations.

Let $(Y_u)_{u \in T_n}$ be i.i.d. spatial increments on the tree $T_n$, where each $Y_u$ has the uniform distribution on the interval $(-1/2, 1/2)$, with variance $\Sigma^2 = 1/12$. The main idea, see the discussion around Equation (5.1) in~\cite{Cai-Holmgren-Janson-Johansson-Skerman:Inversions_in_split_trees_and_conditional_Galton_Watson_trees}, is the introduction of a coupling between an inversion $I$ on $T_n$ and $(Y_u)_{u \in T_n}$, which yields the following comparison:
\[\left|J(T_n)  - \left(I(T_n) - \frac{\Lambda(T_n)}{2}\right)\right| \le 2n,\]
where $J(T_n) = \sum_{v \in T_n} S_v$ and we recall that $S_v$ is the spatial position of the vertex $v$. Still following~\cite[Section 5]{Cai-Holmgren-Janson-Johansson-Skerman:Inversions_in_split_trees_and_conditional_Galton_Watson_trees}, let us define a process $\widehat{R}_n$ on $[0,2n]$ as follows: recall that $\Csp_n$ is the spatial process in \emph{contour order}, then for every $t \in [0,2n]$, if $t$ is an integer, then set $\widehat{R}_n(t) = \Csp_n(t)$, otherwise set
\[\widehat{R}_n(t) =
\begin{cases}
\Csp_n(\lfloor t\rfloor),	&\text{if } C_n(\lfloor t\rfloor) > C_n(\lceil t\rceil),\\
\Csp_n(\lceil t\rceil),	&\text{if } C_n(\lfloor t\rfloor) < C_n(\lceil t\rceil).
\end{cases}\]
In other words, $\widehat{R}_n$ is a step function on $[0,2n]$ which is constant on each interval $[i, i+1)$ with $0 \le i \le 2n-1$, on which it takes the value of the position of either the vertex visited at time $i$ or $i+1$ in the contour order, whichever is the farthest (in graph distance) from the root. It then readily follows that
\[J(T_n) = \frac{1}{2} \int_0^{2n} \widehat{R}_n(t) \d t.\]
Observe that $|\widehat{R}_n(t) - \Csp_n(t)| \le 1/2$ for every $t \in [0,2n]$ so,
\[I(T_n) - \frac{\Lambda(T_n)}{2}
= \frac{1}{2} \int_0^{2n} \Csp_n(t) \d t + O(n)
= n \int_0^1 \Csp_n(2nt) \d t + O(n).\]
Notice that $(n^3/B_n)^{1/2} \gg n$; since $\Sigma^2 = 1/12$, we conclude from Theorem~\ref{thm:convergence_serpent_iid} that
\[\left(\frac{B_n}{12n^3}\right)^{1/2} \left(I(T_n) - \frac{\Lambda(T_n)}{2}\right)
= \int_0^1 \left(\frac{B_n}{12n}\right)^{1/2} \Csp_n(2nt) \d t + o(1)
\cvloi \int_0^1 \Snake_t \d t.\]

\section{Heavy-tailed snakes}
\label{sec:queues_lourdes}

We investigate more precisely in this section the behaviour of $\Hsp_n$ and $\Csp_n$ when the assumption $\P(|Y| \ge (n/B_n)^{1/2}) = o(n^{-1})$ of Theorem~\ref{thm:convergence_serpent_iid} fails. In this case, these processes cannot converge to continuous function since they admit large increments. In fact, they do not converge to functions at all; indeed, with high probability as $n$ becomes large, we may find in the tree $T_n$ vertices, say, $u$, which have a microscopic descendance and such that $|Y_u|$ is very large so the processes $\Hsp_n$ and $\Csp_n$ have a macroscopic increment, almost immediately followed by the opposite increment, which gives rise at the limit to a vertical peak. Nonetheless, as proved by Janson \& Marckert~\cite{Janson-Marckert:Convergence_of_discrete_snakes} they still converge in distribution in the following weaker sense.

In this section, we identify continuous fonctions from $[0,1]$ to $\R$ with their graph, which belong to the space $\K$ of compact subsets of $[0,1] \times \R$, which is a Polish space when equipped with the Hausdorff distance: the distance between two compact sets $A$ and $B$ is
\[d_H(A,B) = \inf\{r > 0 : A \subset B^{(r)} \text{ and } B\subset A^{(r)}\},\]
where $A^{(r)} = \{x \in \R^2 : d(x,A) \le r\}$. Then a sequence of functions $(f_n)_{n \ge 1}$ in $\mathscr{C}([0,1], \R)$ may converge in $\K$ to a limit $K$ which is not the graph of a function; note that if $K$ is the graph of a continuous function, then this convergence is equivalent to the uniform convergence considered previously. The type of limits we shall consider are constructed as follows. Take $f \in \mathscr{C}([0,1], \R)$ and $\Xi$ a collection of points in $[0,1] \times \R$ such that for every $x \in [0,1]$ there exists at most one element $y \in \R$ such that $(x,y) \in \Xi$, and for every $\delta > 0$, the set $\Xi \cap ([0,1] \times (\R \setminus [-\delta, \delta]))$ is finite. We then define
a subset $f \bowtie \Xi \subset [0,1] \times \R$ as 
the union of the graph of $f$ and the following collection of vertical segments: for every point $(x,y) \in \Xi$, we place a vertical segment of length $|y|$ at $(x,f(x))$, directed up or down according to the sign of $y$. Then $f \bowtie \Xi$ belongs to $\K$ and the map $(f,\Xi) \mapsto f \bowtie \Xi$ is measurable so we may take a random function $f$ and a random set $\Xi$ and obtain a random set $f \bowtie \Xi$.

Again, our results focus on the head of the snakes, but they imply the convergence of the entire snakes towards `jumping snakes', see~\cite[Section~3.1]{Janson-Marckert:Convergence_of_discrete_snakes}.

\subsection{The intermediate regime}

In the next result, we investigate the case where $n\cdot \P(|Y| \ge (n/B_n)^{1/2})$ is uniformly bounded. Extracting a subsequence if necessary, we may assume in fact that both tails $n\cdot \P(Y \ge (n/B_n)^{1/2})$ and $n\cdot \P(-Y \ge (n/B_n)^{1/2})$ converge.

\begin{thm}[Convergence to a `hairy snake']\label{thm:convergence_serpent_poilu}
Suppose that $\E[Y]=0$, that $\Sigma^2 \coloneqq \E[Y^2] \in (0, \infty)$ and that there exists $a_+, a_- \in [0,\infty)$ such that $a_+ + a_- > 0$ and
\[\lim_{n \to \infty} n \cdot \P(Y \ge (n/B_n)^{1/2}) = a_+
\qquad\text{and}\qquad
\lim_{n \to \infty} n \cdot \P(-Y \ge (n/B_n)^{1/2}) = a_-.\]
Let $\Xi$ be a Poisson random measure on $[0,1] \times \R$ with intensity $\frac{2\alpha}{\alpha-1} y^{-1-\frac{2\alpha}{\alpha-1}} (a_+ \ind{y > 0} + a_- \ind{y<0}) \d x \d y$ which is independent of the pair $(\Hexc, \Snake)$. Then the convergence in distribution of the sets
\[\left\{\left(\frac{B_n}{n}\right)^{1/2} \Hsp_n(nt) ; t \in [0,1]\right\} \cvloi (\Sigma\cdot\Snake) \bowtie \Xi,\]
holds in $\K$, jointly with~\eqref{eq:cv_GW_Duquesne}. The same holds (jointly) when $\Hsp_n(n\cdot)$ is replaced by $\Csp_n(2n\cdot)$.
\end{thm}

\begin{figure}[!ht] \centering
\includegraphics[width=.45\linewidth]{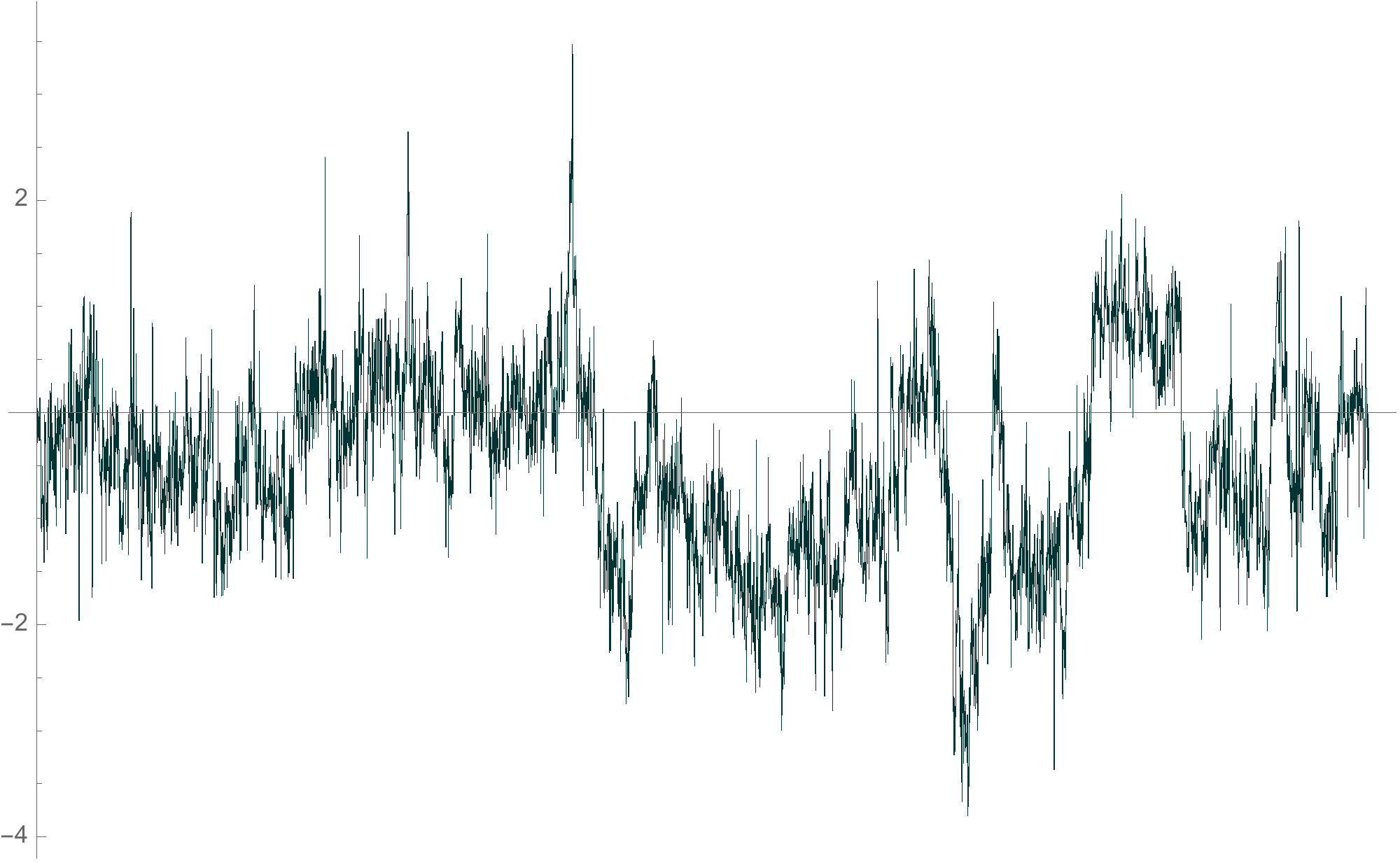}
\qquad
\includegraphics[width=.45\linewidth]{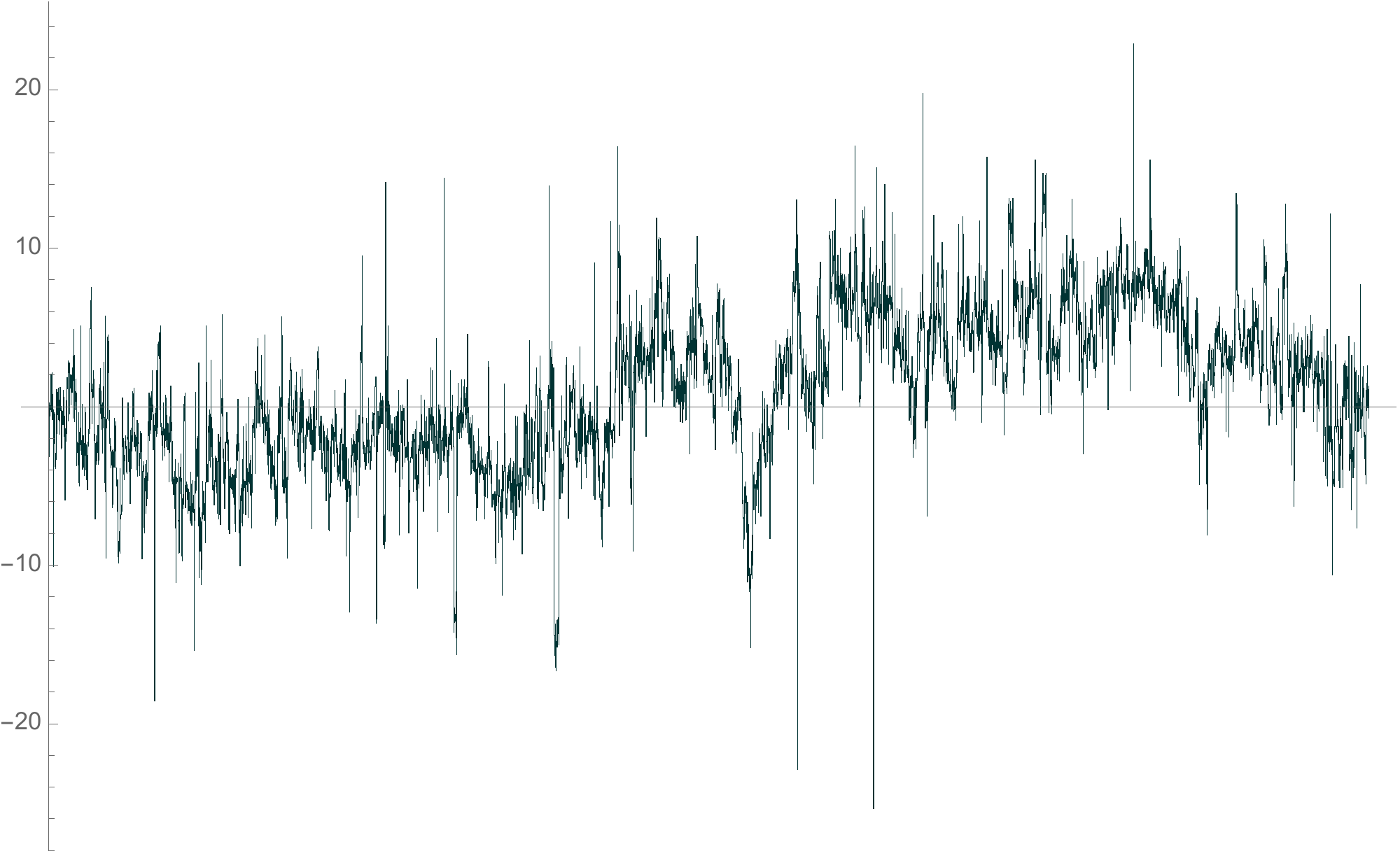}
\caption{Two instances of the spatial height process $H_n(n \cdot)$ associated with the height process of Figure~\ref{fig:Levy_hauteur} where in both cases, $Y$ is symmetric and such that $\lim_{n \to \infty} n \cdot \P(Y \ge (n/B_n)^{1/p}) = 1$; on the left, $p=2$ and on the right, $p=\numprint{0.6}$.}
\label{fig:serpents_poilus}
\end{figure}

The intuition behind this result is that, as opposed to Theorem~\ref{thm:convergence_serpent_iid}, we can find here vertices $u$ of $T_n$ such that $|Y_u|$ is macroscopic, and these points lead to the peaks given by $\Xi$ at the limit. Indeed, for every $c > 0$, we have
\[\left(\frac{cn}{B_{cn}}\right)^{1/2} = c^{\frac{\alpha-1}{2\alpha}} \left(\frac{n}{B_n}\right)^{1/2} \left(\frac{n^{-1/\alpha} B_n}{(cn)^{-1/\alpha} B_{cn}}\right)^{1/2},\]
and the very last term converges to $1$ since $n^{-1/\alpha} B_n$ is slowly varying at infinity. Therefore, for every $y > 0$ fixed, we have
\[y \left(\frac{n}{B_n}\right)^{1/2} \enskip\mathop{\sim}^{}_{n \to \infty}\enskip \left(\frac{y^{\frac{2\alpha}{\alpha-1}}n}{B_{y^{\frac{2\alpha}{\alpha-1}}n}}\right)^{1/2}.\]
Observe that under the assumption of Theorem~\ref{thm:convergence_serpent_poilu}, we have
\[\Pr{Y > \left(\frac{y^{\frac{2\alpha}{\alpha-1}}n}{B_{y^{\frac{2\alpha}{\alpha-1}}n}}\right)^{1/2}} \enskip\mathop{\sim}^{}_{n \to \infty}\enskip a_+ y^{-\frac{2\alpha}{\alpha-1}} n^{-1}.\]
For every $\varepsilon > 0$, the same relation holds when each occurence (on both sides) of $y^{\frac{2\alpha}{\alpha-1}}$ is multiplied by $1+\varepsilon$ or by $1-\varepsilon$, which enables to conclude (the second limit is obtained by a symmetric argument) that under the assumption of Theorem~\ref{thm:convergence_serpent_poilu}, we have
\[\lim_{n \to \infty} n \cdot \Pr{Y > y \left(\frac{n}{B_n}\right)^{1/2}} = a_+ y^{-\frac{2\alpha}{\alpha-1}},
\qquad\text{and}\qquad
\lim_{n \to \infty} n \cdot \Pr{-Y < y \left(\frac{n}{B_n}\right)^{1/2}} = a_- y^{-\frac{2\alpha}{\alpha-1}}.\]
Since, conditional on $T_n$, the cardinal $\#\{u \in T_n : Y_u > y (n/B_n)^{1/2}\}$ has the binomial distribution with parameters $n$ and $\P(Y > y (n/B_n)^{1/2})$, this number is asymptotically Poisson distributed with rate $a_+ y^{-\frac{2\alpha}{\alpha-1}}$, which indeed corresponds to $\Xi$, provided that the locations are asymptotically uniformly distributed in the tree.

As for Theorem~\ref{thm:convergence_serpent_iid}, let us decompose the proof into several steps, following closely the argument of the proof of Theorem~5 in \cite{Janson-Marckert:Convergence_of_discrete_snakes}.

\subsubsection{Contribution of the small jumps}\label{sec:thm_serpent_poilu_step1}
As in the proof of Theorem~\ref{thm:convergence_serpent_iid}, let us treat separately the large and small increments: put $b_n = (n^2/B_n)^{\frac{\alpha-1}{4\alpha} + \varepsilon}$ for some $\varepsilon > 0$ small to be tuned. For every vertex $u \in T_n$, let $Y_u' = Y_u \ind{|Y_u| \le b_n}$ and $Y_u'' = Y_u \ind{|Y_u| > b_n}$, define then $\Hsp_n\vphantom{H}'$ and $\Hsp_n\vphantom{H}''$ the spatial processes in which the increments $Y_u$ are replaced by $Y_u'$ and $Y_u''$ respectively, so $\Hsp_n = \Hsp_n\vphantom{H}' + \Hsp_n\vphantom{H}''$. Observe that the only difference with Theorem~\ref{thm:convergence_serpent_iid} is that our assumption now implies $\P(|Y| \ge y) = O(y^{-\frac{2\alpha}{\alpha(1+\delta)-1}})$, whereas the big $O$ was a small $o$ there. Actually, the arguments used to control the small jumps in Section~\ref{sec:thm_principal_step4} and~\ref{sec:thm_principal_step5} only requires a big $O$ so we conclude that, as there, we have
\[\left(\frac{B_n}{n} H_n(n t), \left(\frac{B_n}{n \Sigma^2}\right)^{1/2} \Hsp_n\vphantom{H}'(nt)\right)_{t \in [0,1]} \cvloi (\Hexc_t, \Snake_t)_{t \in [0,1]}.\]

\subsubsection{Contribution of the medium-large jumps}\label{sec:thm_serpent_poilu_step2}
We next claim that we have
\begin{equation}\label{eq:serpent_poilu_grands_sauts}
\left\{\left(\frac{B_n}{n}\right)^{1/2} \Hsp_n\vphantom{H}''(nt) ; t \in [0,1]\right\} \cvloi 0 \bowtie \Xi,
\end{equation}
in $\K$. Let us approximate both sides. First note that the intensity measure of $\Xi$ explodes at the $x$ axis, but for every $\eta > 0$, the restricted Poisson random measure $\Xi^{\eta} = \Xi \cap ([0,1] \times (\R\setminus [-\eta, \eta]))$ is almost surely finite. Clearly, the Hausdorff distance between $0 \bowtie \Xi$ and $0 \bowtie \Xi^{\eta}$ is at most $\eta$. Similarly, let us truncate further $Y_u''$ by setting $Y_u^{\eta} = Y_u \ind{|Y_u| > \eta (n/B_n)^{1/2}}$ and then define $\Hsp_n\vphantom{H}^{\eta}$ as the spatial process in which the increments $Y_u$ are replaced by $Y_u^{\eta}$. Recall the event $E_n$ that $T_n$ contains two vertices, say $u$ and $v$, such that $u$ is an ancestor of $v$ and both $|Y_u| > b_n$ and $|Y_v| > b_n$. Then again, $\P(E_n)$ converges to $0$ since the argument used in Section~\ref{sec:thm_principal_step3} only requires a big $O$ in the tail of $Y$. Furthermore, on the complement event, we have $\max_{0 \le i \le n} |\Hsp_n\vphantom{H}''(i)| = \max_{u \in T_n} |Y_u''|$ and $\max_{0 \le i \le n} |\Hsp_n\vphantom{H}^{\eta}(i)| = \max_{u \in T_n} |Y_u^{\eta}|$. The latter is either $\max_{u \in T_n} |Y_u''|$ or $0$ in the case $\max_{u \in T_n} |Y_u''| \le \eta (n / B_n)^{1/2}$ so we conclude that on the event $E_n^c$ whose probability tends to $1$, we have
\[\max_{0 \le i \le n} |\Hsp_n\vphantom{H}''(i)| - \max_{0 \le i \le n} |\Hsp_n\vphantom{H}^{\eta}(i)| \le \eta \left(\frac{n}{B_n}\right)^{1/2}.\]
Therefore, in order to prove~\eqref{eq:serpent_poilu_grands_sauts}, it suffices to prove that for every $\eta > 0$, it holds that
\begin{equation}\label{eq:serpent_poilu_tres_grands_sauts}
\left\{\left(\frac{B_n}{n}\right)^{1/2} \Hsp_n\vphantom{H}^{\eta}(nt) ; t \in [0,1]\right\} \cvloi 0 \bowtie \Xi^{\eta}.
\end{equation}

\subsubsection{A discrete Poisson random measure}\label{sec:thm_serpent_poilu_step3}
Let $\varnothing = u_0 < u_1<  \dots < u_n$ be the vertices of $T_n$ listed in lexicographical order, and let $1 \le k_1 < \dots < k_{N_n} \le n$ be the indices of those vertices $u$ of $T_n$ for which $Y_u^{\eta}$ is non-zero, or otherwise said $|Y_u| > \eta (n/B_n)^{1/2}$. Let $\nu_\eta(\d y) = \frac{2\alpha}{\alpha-1} y^{-1-\frac{2\alpha}{\alpha-1}} (a_+ \ind{y > \eta} + a_- \ind{y < -\eta})\d y$ so $\d x\nu_\eta(\d y)$ is the intensity of $\Xi^{\eta}$. Then the discussion just after the statement of the theorem shows that $N_n$ has the binomial distribution with parameters $n$ and $\P(|Y| > \eta (n/B_n)^{1/2})$ which converges to the Poisson distribution with parameter $\nu_\eta(\R) \in (0,\infty)$. Furthermore, conditional on $N_n$, 
the values $(Y_{u_{k_i}}^{\eta})_{1 \le i \le N_n}$ have the same distribution as i.i.d. copies of $Y$ conditioned on $|Y| > \eta (n/B_n)^{1/2}$, the indices $(k_i)_{1 \le i\le N_n}$ have the uniform distribution amongst the ranked $N_n$-tuples in $\{1, \dots, n\}$, and they are independent of the values $(Y_{u_{k_i}}^{\eta})_{1 \le i \le N_n}$. All the other $Y_u$'s are null.

Since $N_n$ is bounded in probability, the $N_n$-tuple $(k_i)_{1 \le i\le N_n}$ is well approximated as $n \to \infty$ by i.d.d. random times on $\{1, \dots, n\}$. Furthermore, one easily checks from their tail behaviour that the random variables $(B_n/n)^{1/2} Y_{u_{k_i}}^{\eta}$ converge in distribution as $n \to \infty$ towards i.i.d. random variables sampled from the probability $\nu_\eta(\cdot) / \nu_\eta(\R)$. Therefore, the set $\xi_n^\eta = \{(n^{-1} k_i, (B_n/n)^{1/2} Y_{u_{k_i}}^{\eta}) ; 1 \le i \le N_n\} \subset [0,1]\times \R$.  converges in law in $\K$ towards the random set $\xi_\infty^\eta = \{(U_i, X_i) ; 1 \le i \le N\}$ where $(U_i)_{i \ge 1}$ are i.i.d. uniformly distributed on $[0,1]$, $(X_i)_{i \ge 1}$ are i.i.d. with law $\nu_\eta(\cdot) / \nu_\eta(\R)$, $N$ has the Poisson distribution with parameter $\nu_\eta(\R)$, and all are independent. Then $\xi_\infty^\eta$ has the law of $\Xi^{\eta}$. One easily check that the mapping $\xi \mapsto 0 \bowtie \xi$ is continuous $\Xi$-almost surely, so we conclude that
\[0 \bowtie \xi_n \cvloi 0 \bowtie \Xi^{\eta}.\]
in $\K$.

\subsubsection{Contribution of the very large jumps}\label{sec:thm_serpent_poilu_step4}
In order to prove that~\eqref{eq:serpent_poilu_tres_grands_sauts}, and therefore~\eqref{eq:serpent_poilu_grands_sauts}, holds, in only remains to prove that
\[d_H\left(\left\{\left(\frac{B_n}{n}\right)^{1/2} \Hsp_n\vphantom{H}^{\eta}(nt) ; t \in [0,1]\right\}, 0 \bowtie \xi_n^\eta\right) \cvproba 0.\]
We implicitly work conditional on the event $E_n^c$ so there is at most one non zero value of $Y_u^{\eta}$ along each branch of $T_n$. In this case, the process $(B_n / n) \cdot \Hsp_n\vphantom{H}^{\eta}$ can be described at follows: it is null until time $k_1 - 1$, then it moves to a random value $Y_{u_{k_1}}$ at time $k_1$, it stays at this value for a time given by the total progeny of $u_{k_1}$ before going back to zero where it stays until time $k_2$ and so on. On the other hand, $0 \bowtie \xi_n^\eta$ is constructed by putting value $0$ for every time $t \in [0,1]$ except the $k_i$'s where we place vertical peaks given by the $Y_{u_{k_i}}$'s. Then the previous convergence is an easy consequence of the following lemma which extends \cite[Lemma~8]{Janson-Marckert:Convergence_of_discrete_snakes}:

\begin{lem}\label{lem:descendance_sommet_unif}
Let $v_n$ be uniformly distributed in $T_n$, and let $D(v_n)$ be its number of descendants, then $D(v_n) / n$ converges in probability to $0$.
\end{lem}

Indeed, since $N_n$ is bounded in probability and our vertices $(u_{k_i})_{1 \le i \le N_n}$ are uniformly distributed (and conditioned to be different and to lie on different branches, but this occurs with high probability), it follows that their progeny are all small compared to $n$, so as $n \to \infty$, the process $(B_n / n) \cdot \Hsp_n\vphantom{H}^{\eta}$ indeed goes back almost immediately to $0$ after reaching a high value, as for $0 \bowtie \xi_n^\eta$.

\begin{proof}[Proof of Lemma~\ref{lem:descendance_sommet_unif}]
Let us condition on $T_n$:
\[\Esc{D(v_n)}{T_n} = \frac{1}{n+1} \sum_{v \in T_n} \sum_{w \in T_n} \ind{w \text{ is an ancestor of } v}
= \frac{1}{n+1} \sum_{w \in T_n} |v|
= \frac{1}{n+1} \Lambda(T_n),\]
where we recall the notation $\Lambda(T_n)$ for the total path length of $T_n$, which is of order $n^2/B_n$ according to~\eqref{eq:convergence_total_path_length}. We conclude from the Markov inequality that for every $\varepsilon, C > 0$, we have
\[\Pr{D(v_n) > \varepsilon n} \le \Pr{\Lambda(T_n) > C (n+1)^2/B_n} + C/B_n,\]
which converges to $0$ when letting first $n \to \infty$ and then $C \to \infty$.
\end{proof}

\subsubsection{Combining small and large jumps}\label{sec:thm_serpent_poilu_step5}
The proof is not finished! We have shown that
\[\left(\frac{B_n}{n} H_n(n t), \left(\frac{B_n}{n}\right)^{1/2} \Hsp_n\vphantom{H}'(nt)\right)_{t \in [0,1]} \cvloi (\Hexc_t, \Sigma^2\cdot \Snake_t)_{t \in [0,1]},\]
and
\[\left\{\left(\frac{B_n}{n}\right)^{1/2} \Hsp_n\vphantom{H}''(nt) ; t \in [0,1]\right\} \cvloi 0 \bowtie \Xi,\]
and yet, although, if for $f \in \mathscr{C}([0,1],\R)$ and $B \in \K$, we put $A \dotplus B = \{(x, y_A + y_B) ; (x,y_A) \in A, (x,y_B) \in B\}$, then this addition is continuous, we cannot directly conclude that
\[\left\{\left(\frac{B_n}{n}\right)^{1/2} (\Hsp_n\vphantom{H}'(nt)+\Hsp_n\vphantom{H}''(nt)) ; t \in [0,1]\right\} \cvloi (\Sigma^2\cdot \Snake_t) \bowtie \Xi = (\Sigma^2\cdot \Snake_t) \dotplus (0 \bowtie \Xi),\]
because the previous convergences in distribution may not hold simultaneously. Indeed, the processes $\Hsp_n\vphantom{H}'$ and $\Hsp_n\vphantom{H}''(nt)$ are not independent since each $Y_u$ either contributes to one or to the other.

We create independence by resampling the $Y_u$'s which contribute to $\Hsp_n\vphantom{H}''(nt)$ as follows: let $(Z_i)_{i \ge 1}$ be i.i.d. copies of $Y \ind{|Y| \le b_n}$ independent of the rest and put
\[\widetilde{Y}_i = Y_i \ind{|Y| \le b_n} + Z_i \ind{|Y| > b_n},\]
for each $1 \le i \le n$. Now the processes $\tildeHsp_n$ and $\Hsp_n\vphantom{H}''(nt)$ are independent, and furthermore, the error between $\Hsp_n\vphantom{H}'$ and $\tildeHsp_n$ comes from those $Y_u$'s for which $|Y_u| > b_n$; on the event $E_n^c$, there exists at most one such $u$ on each branch and therefore $\max_{0 \le i \le n} |\tildeHsp_n(i) - \Hsp_n\vphantom{H}'(i)| \le b_n = o(n/B_n)^{1/2}$ since each $\widetilde{Y}_i$ and each $Y_i'$ belongs to $[0, b_n]$. We conclude that
\[\left(\left(\frac{B_n}{n}\right)^{1/2} \Hsp_n\vphantom{H}'(nt), \left(\frac{B_n}{n}\right)^{1/2} \tildeHsp_n(nt)\right)_{t \in [0,1]} \cvloi (\Sigma^2\cdot \Snake_t, \Sigma^2\cdot \Snake_t)_{t \in [0,1]},\]
and the proof of Theorem~\ref{thm:convergence_serpent_poilu} is now complete.

\subsection{The strong heavy tail regime}
We finally investigate the regime where the tails of $Y$ are much stronger than what requires Theorem~\ref{thm:convergence_serpent_iid}. In this case, the extreme values dominate the small ones and the snake disappears at the limit, and only the vertical peaks remain, see Figure~\ref{fig:serpents_poilus} for a comparison with the previous case.

\begin{thm}[Convergence to a `flat hairy snake']\label{thm:convergence_serpent_poilu_plat}
Fix $p \in (0,2]$ and suppose that $\E[Y]=0$. Assume that there exists $\varrho \in [0,1]$ and two slowly varying functions at infinity $L^+$ and $L^-$ such that if $L = L^+ + L^-$, then as $x \to \infty$, the ratios $L^+(x)/L(x)$ and $L^-(x)/L(x)$ converge respectively to $\varrho$ and $1-\varrho$, and furthermore
\[n \cdot \P(Y \ge (n/B_n)^{1/p} L^+(n/B_n)) \cv 1
\qquad\text{and}\qquad
n \cdot \P(-Y \ge (n/B_n)^{1/p} L^-(n/B_n)) \cv 1,\]
If $p=2$, assume also that the function $L$ tends to infinity. Let $\Xi$ be a Poisson random measure on $[0,1] \times \R$ independent of the pair $\Snake$, with intensity
\[\frac{p\alpha}{\alpha-1} y^{-1-\frac{p\alpha}{\alpha-1}} \bigg(\varrho^{1+\frac{p\alpha}{\alpha-1}} \ind{y > 0} + (1-\varrho)^{1+\frac{p\alpha}{\alpha-1}} \ind{y<0}\bigg) \d x \d y.\]
Then the convergence in distribution of the sets
\[\left\{\frac{B_n^{1/p}}{n^{1/p} L(n/B_n)} \Hsp_n(nt) ; t \in [0,1]\right\} \cvloi 0 \bowtie \Xi,\]
holds in $\K$, jointly with~\eqref{eq:cv_GW_Duquesne}. The same holds (jointly) when $\Hsp_n(n\cdot)$ is replaced by $\Csp_n(2n\cdot)$.
\end{thm}

In the case $L^+(x) \to c_+ \in [0, \infty)$ and $L^-(x) \to c_- \in [0, \infty)$, the assumption reads
\[n \cdot \P(Y \ge (n/B_n)^{1/p}) \cv a_+
\qquad\text{and}\qquad
n \cdot \P(-Y \ge (n/B_n)^{1/p}) \cv a_-,\]
where $a_{+/-} = (c_{+/-})^{\frac{p\alpha}{\alpha-1}}$, and then the conclusion reads
\[\left\{\left(\frac{B_n}{n}\right)^{1/p} \Hsp_n(nt) ; t \in [0,1]\right\} \cvloi 0\bowtie \Xi,\]
where $\Xi$ has intensity $\frac{p\alpha}{\alpha-1} y^{-1-\frac{p\alpha}{\alpha-1}} (a_+ \ind{y > 0} + a_- \ind{y < 0}) \d x \d y$, which recovers~\cite[Theorem~6]{Janson-Marckert:Convergence_of_discrete_snakes}.

\begin{rem}
Recall if $\E[Y] = m \ne 0$, then $(B_n/n)\cdot \E[\Hsp_n(n\cdot)\mid T_n]$ converges to $m \cdot \Hexc$ so the previous result still holds in this case for $p < 1$, and when $p=1$ and $L^+$ and $L^-$ both converge, then 
\[\left\{\left(\frac{B_n}{n}\right)^{1/p} \Hsp_n(nt) ; t \in [0,1]\right\} \cvloi m \cdot \Hexc \bowtie \Xi,\]
in $\K$, jointly with~\eqref{eq:cv_GW_Duquesne}, where $\Hexc$ and $\Xi$ are independent, and the same holds (jointly) when $\Hsp_n(n\cdot)$ is replaced by $\Csp_n(2n\cdot)$.
\end{rem}

\begin{proof}
Since $L^{+/-}$ are slowly varying, we have $(n/B_n)^{-\theta} \ll L^{+/-}(n/B_n) \ll (n/B_n)^\theta$ for every $\theta > 0$ so the tails of $Y$ satisfy now $\P((Y)_{+/-} > y) = o(y^{-\frac{\alpha(p-\theta)}{\alpha(1+\delta)-1}})$ for every $\delta,\theta > 0$. As usual, let us cut the increments: put $b_n = (n^2/B_n)^{\frac{\alpha-1}{2p\alpha} + \varepsilon}$ for some $\varepsilon > 0$ small to be tuned. For every vertex $u \in T_n$, let $Y_u' = Y_u \ind{|Y_u| \le b_n}$ and $Y_u'' = Y_u \ind{|Y_u| > b_n}$, define then $\Hsp_n\vphantom{H}'$ and $\Hsp_n\vphantom{H}''$ the spatial processes in which the increments $Y_u$ are replaced by $Y_u'$ and $Y_u''$ respectively, so $\Hsp_n = \Hsp_n\vphantom{H}' + \Hsp_n\vphantom{H}''$. Similarly as in the proof of Theorem~\ref{thm:convergence_serpent_iid} and Theorem~\ref{thm:convergence_serpent_non_centre}, the exponent in $b_n$ matches that in the tails of $Y$. Therefore, again, it hods that
\[\frac{n^2}{B_n} \Pr{|Y| > b_n}^2 \ll \left(\frac{n^2}{B_n}\right)^{1-(\frac{\alpha-1}{2\alpha p} + \varepsilon) \frac{\alpha(p-\theta)}{\alpha(1+\delta)-1}},\]
and the exponent is negative for $\delta, \theta$ small enough. The event $E_n$ that $T_n$ contains two vertices, say $u$ and $v$, such that $u$ is an ancestor of $v$ and both $|Y_u| > b_n$ and $|Y_v| > b_n$, thus has a probability tending to $0$. Then the arguments used in Section~\ref{sec:thm_serpent_poilu_step2},~\ref{sec:thm_serpent_poilu_step3} and~\ref{sec:thm_serpent_poilu_step4} extend readily to prove that
\[\left\{\frac{B_n^{1/p}}{n^{1/p} L(n/B_n)} \Hsp_n\vphantom{H}''(nt) ; t \in [0,1]\right\} \cvloi 0 \bowtie \Xi,\]
and it only remains to prove the convergence
\[\left(\frac{B_n^{1/p}}{n^{1/p} L(n/B_n)} \Hsp_n\vphantom{H}'(nt)\right)_{t \in [0,1]} \cvproba 0.\]
We first consider the average displacement induced by these small jumps. Let $m_n = \E[Y'] = -\E[Y'']$, so
\[\frac{B_n^{1/p}}{n^{1/p} L(n/B_n)} \Esc{\Hsp_n\vphantom{H}'(n \cdot)}{T_n}
= \frac{B_n^{-(p-1)/p}}{n^{-(p-1)/p} L(n/B_n)} m_n \cdot \frac{B_n}{n} H_n(n \cdot).\]
From the tail behaviour of $Y$ we get:
\[|m_n| \le \E[|Y''|]
\le b_n \Pr{|Y| > b_n} + \int_{b_n}^\infty \Pr{|Y| > y} \d y
= O\left(b_n^{1-\frac{\alpha(p-\theta)}{\alpha(1+\delta)-1}}\right).\]
Exactly as in the proof of Theorem~\ref{thm:convergence_serpent_iid}, note that for $\varepsilon, \delta, \theta$ sufficiently small, we have
\[\left(\frac{\alpha-1}{2p\alpha} + \varepsilon\right) \left(1-\frac{\alpha(p-\theta)}{\alpha(1+\delta)-1}\right)
> - \frac{\alpha+1}{2\alpha},\]
and this bound gets tighter as $\varepsilon, \delta, \theta$ get closer to $0$. Then, since
\[\frac{\alpha+1}{2\alpha} < \frac{p-1}{p} < \frac{\alpha+1}{\alpha},\]
we conclude that for $\varepsilon, \delta, \theta$ sufficiently small,
\[\frac{n^{(p-1)/p}}{B_n^{(p-1)/p} L(n/B_n)} \left(b_n^{1-\frac{\alpha(p-\theta)}{\alpha(1+\delta)-1}}\right)
\le n^{2 (\frac{\alpha-1}{2p\alpha} + \varepsilon)(1-\frac{\alpha(p-\theta)}{\alpha(1+\delta)-1}) + \frac{p-1}{p-\theta}} \cdot B_n^{- (\frac{\alpha-1}{2p\alpha} + \varepsilon)(1-\frac{\alpha(p-\theta)}{\alpha(1+\delta)-1}) - \frac{p-1}{p-\theta}}\]
converges to $0$ as $n \to \infty$ since both exponents are negative. Therefore the process $\E[\Hsp_n\vphantom{H}'(n\cdot)\mid T_n]$ suitably rescaled converges in probability to $0$ and we focus for the rest of the proof on the centred process $\Hspc_n\vphantom{H}'(n\cdot) = \Hsp_n\vphantom{H}'(n\cdot) - \E[\Hsp_n\vphantom{H}'(n\cdot)\mid T_n]$.

Similarly, with the notation from Section~\ref{sec:serpents}, for $s,t \in [0,1]$, we have
\begin{align*}
&\Esc{\left(\frac{B_n^{1/p}}{n^{1/p} L(n/B_n)} |\Hspc_n\vphantom{H}'(nt) - \Hspc_n\vphantom{H}'(ns)|\right)^q}{A_n}
\\
&\le C_q \frac{B_n^{q/p}}{n^{q/p} L(n/B_n)^q} \left(\frac{n}{B_n} |t-s|^\gamma \Es{|\mathsf{Y}'|^q} + \left(\frac{n}{B_n} |t-s|^\gamma\right)^{q/2} \Es{|\mathsf{Y}'|^2}^{q/2}\right)
\\
&\le C_q \left(L\left(\frac{n}{B_n}\right)^{-q} \left(\frac{B_n}{n}\right)^{\frac{q}{p}-1} \left(\frac{n^2}{B_n}\right)^{q(\frac{\alpha-1}{p\alpha} + \varepsilon)} |t-s|^\gamma + \left[L\left(\frac{n}{B_n}\right)^{-2} \left(\frac{B_n}{n}\right)^{\frac{2}{p} - 1} \Es{|Y'|^2}\right]^{q/2} |t-s|^{q\gamma/2}\right).
\end{align*}
The first term in the last line is controlled as in the proof of Theorem~\ref{thm:convergence_serpent_iid} and Theorem~\ref{thm:convergence_serpent_non_centre}: the slowly varying function will not cause any trouble, and the factor $1/p$ in the exponent in $b_n$ compensate the fact that we now rescale by $(n/B_n)^{1/p}$ instead of $(n/B_n)^{1/2}$ for Theorem~\ref{thm:convergence_serpent_iid} and similar calculations as in the proof of the latter show that this first term is bounded by $|t-s|^\gamma$ times $n$ raised to a power which converges to $-\infty$ as $q \to \infty$. The only change compared to the proof of Theorem~\ref{thm:convergence_serpent_iid} is that, as for Theorem~\ref{thm:convergence_serpent_non_centre}, we may not have $\E[|Y|^2] < \infty$. Still,
\[\Es{|Y'|^2}
= 2 \int_0^{b_n} y \Pr{|Y| > y} \d y
= O\left(\int_1^{b_n} y^{1-\frac{\alpha(p-\theta)}{\alpha(1+\delta)-1}} \d y\right).\]
Either $\frac{\alpha(p-\theta)}{\alpha-1} > 2$, in which case $1-\frac{\alpha(p-\theta)}{\alpha(1+\delta)-1} < -1$ for $\delta$ small enough and the integral converges, or $\frac{\alpha(p-\theta)}{\alpha-1} \le 2$, in which case, $1 - \frac{\alpha(p-\theta)}{\alpha(1+\delta)-1} > -1$, so $\E[|Y'|^2] = O(b_n^{2-\frac{\alpha(p-\theta)}{\alpha(1+\delta)-1}})$ and then
\[\left(\frac{B_n}{n}\right)^{\frac{2}{p} - 1} \Es{|Y'|^2} 
= O\left(B_n^{\frac{2}{p} - 1 -(\frac{\alpha-1}{2p\alpha} + \varepsilon)(2-\frac{\alpha(p-\theta)}{\alpha(1+\delta)-1})} n^{1 - \frac{2}{p} + 2(\frac{\alpha-1}{2p\alpha} + \varepsilon)(2-\frac{\alpha(p-\theta)}{\alpha(1+\delta)-1})}\right).\]
Now for every $\eta > 0$, it holds that $B_n = o(n^{\frac{1+\eta}{\alpha}})$, so finally
\[\left(\frac{B_n}{n}\right)^{\frac{2}{p} - 1} \Es{|Y'|^2} 
= o\left(n^{\frac{1+\eta}{\alpha} (\frac{2}{p} - 1 -(\frac{\alpha-1}{2p\alpha} + \varepsilon)(2-\frac{\alpha(p-\theta)}{\alpha(1+\delta)-1})) + 1 - \frac{2}{p} + 2(\frac{\alpha-1}{2p\alpha} + \varepsilon)(2-\frac{\alpha(p-\theta)}{\alpha(1+\delta)-1})}\right).\]
Note that we got rid of the term $L(n/B_n)^{-2}$ but its contribution is negligible. For $\eta = \varepsilon = \delta = \theta = 0$, the preceding exponent reduces to 
\[\frac{1}{\alpha} \left(\frac{2}{p} - 1 - \frac{\alpha-1}{2p\alpha} \left(2-\frac{\alpha p}{\alpha-1}\right)\right) + 1 - \frac{2}{p} + \frac{\alpha-1}{p\alpha} \left(2-\frac{\alpha p}{\alpha-1}\right)
= \frac{2-\alpha(p+2)}{4\alpha^2 p},\]
which is negative for any $p > 0$ and $\alpha > 1$. Then the exponent is still negative for $\eta, \varepsilon, \delta, \theta > 0$ small enough. In almost all cases, $(B_n/n)^{\frac{2}{p} - 1} \E[|Y'|^2]$ is thus bounded by a negative power of $n$, the only case where it fails is for $p=2$, in which case $\frac{\alpha(p-\theta)}{\alpha-1} > 2$ for $\theta$ small enough, so $(B_n/n)^{\frac{2}{p} - 1} \E[|Y'|^2]$ is uniformly bounded, but this in not an issue. We conclude as in the proof of Theorem~\ref{thm:convergence_serpent_iid} and Theorem~\ref{thm:convergence_serpent_non_centre} that for $q$ large enough,
\[\sup_{n \to \infty} \Esc{\left(\frac{B_n^{1/p}}{n^{1/p} L(n/B_n)} |\Hspc_n\vphantom{H}'(nt) - \Hspc_n\vphantom{H}'(ns)|\right)^q}{A_n}
\le C_q |t-s|^2,\]
so the process $\frac{B_n}{n} \Hspc_n\vphantom{H}'(n\cdot)$ is tight, and the bounds applied with $s=0$ and $t \in [0,1]$ fixed show that the one-dimensional marginals converge in distribution to $0$ so the whole process converges in distribution to the null process, which completes the proof.
\end{proof}


{\small
\newcommand{\etalchar}[1]{$^{#1}$}

}

\end{document}